\documentclass[reqno]{amsart}

\usepackage{amscd,amsmath,amssymb,amsfonts,bm}
\usepackage{stmaryrd}
\usepackage{bbold}
\usepackage{times}
\usepackage[all]{xy}
\usepackage{mathrsfs}
\usepackage{graphicx}

\usepackage{amsthm}

\theoremstyle{plain}
\newtheorem{thm}{Theorem}
\newtheorem{lem}[thm]{Lemma}
\newtheorem{cor}[thm]{Corollary}
\newtheorem{prop}[thm]{Proposition}

\theoremstyle{definition}

\newtheorem*{claim}{Claim}

\newtheorem*{ack}{Acknowledgment}
\newtheorem*{notat}{Notations}

\theoremstyle{remark}

\newcommand{\Ker}{\operatorname{Ker}}
\renewcommand{\Im}{\operatorname{Im}}

\newcommand{\End}{\operatorname{End}}
\newcommand{\Aut}{\operatorname{Aut}}

\newcommand{\Spec}{\operatorname{Spec}}
\newcommand{\Fract}{\operatorname{Fract}}

\newcommand{\GL}{\operatorname{GL}}

\newcommand{\Id}{\operatorname{Id}}

\newcommand{\ch}{\operatorname{char}}

\newcommand{\onto}{\twoheadrightarrow}

\newcommand{\into}{\hookrightarrow}
\newcommand{\tto}{\longrightarrow}
\newcommand{\iso}{\stackrel{\sim}{\tto}}

\renewcommand{\bar}[1]{\overline{#1}}

\newcommand{\til}[1]{\widetilde{#1}}

\newcommand{\Mat}{\operatorname{M}}

\DeclareMathOperator{\Rat}{Rat}

\renewcommand{\k}{\mathbb{k}}
\renewcommand{\phi}{\varphi}

\newcommand{\GSpec}{G\text{-}\!\Spec}
\newcommand{\NSpec}{N\text{-}\!\Spec}
\newcommand{\GRat}{G\text{-}\!\Rat}

\newcommand{\cO}{\mathcal{O}}

\newcommand{\cC}{\mathcal{C}}

\newcommand{\fp}{\mathfrak{p}}

\newcommand{\fb}{\mathfrak{b}}

\newcommand{\Gm}{\mathbb{G}_{\text{\rm m}}}

\newcommand{\ZZ}{\mathbb{Z}}

\newcommand{\cen}{\mathcal{Z}}

\newcommand{\bq}{\mathbf{q}}
\newcommand{\bfm}{\mathbf{m}}

\newcommand{\bV}{\mathbf{V}}
\newcommand{\bI}{\mathbf{I}}
\newcommand{\bga}{\gamma'}
\newcommand{\bgr}{\bga_{\text{rat}}}
\newcommand{\gar}{\gamma_{\text{rat}}}
\newcommand{\pir}{\pi_{\text{rat}}}

\newcommand{\byG}{\!\!:\!\! G}
\newcommand{\byN}{\!\!:\!\! N}
\newcommand{\byH}{\!\!:\!\! H}
\newcommand{\byGa}{\!\!:\!\! \Gamma}

\newcommand{\upin}{\text{\ \rotatebox{90}{$\in$}\ }}
\newdir{ >}{{}*!/-5pt/\dir{>}}

\begin{document}

%
%

\title[Algebraic group actions on noncommutative spectra]%
{Algebraic group actions on noncommutative spectra}

\title[Algebraic group actions on noncommutative spectra]%
{Algebraic group actions on noncommutative spectra}

\author{Martin Lorenz}

\address{Department of Mathematics, Temple University,
    Philadelphia, PA 19122}

\email{lorenz@temple.edu}

\urladdr{http://www.math.temple.edu/$\stackrel{\sim}{\phantom{.}}$lorenz}

\thanks{Research of the author supported in part by NSA Grant H98230-07-1-0008}



\maketitle

\begin{abstract}
Let $G$ be an affine algebraic group and let $R$ be an associative
algebra with a rational action of $G$ by algebra automorphisms. We
study the induced $G$-action on the set $\Spec R$ of all prime
ideals of $R$, viewed as a topological space with the
Jacobson-Zariski topology, and on the subspace $\Rat R \subseteq
\Spec R$ consisting of all rational ideals of $R$. Here, a prime
ideal $P$ of $R$ is said to be rational if the extended centroid
$\cC(R/P)$ is equal to the base field. Our results generalize work
of M{\oe}glin \& Rentschler and Vonessen to arbitrary associative
algebras while also simplifying some of the earlier proofs.

The map $P \mapsto  \bigcap_{g \in G} g.P$ gives a surjection from
$\Spec R$ onto the set $\GSpec R$ of all $G$-prime ideals of $R$.
The fibres of this map yield the so-called $G$-stratification of
$\Spec R$ which has played a central role in the recent
investigation of algebraic quantum groups, in particular in the work
of Goodearl and Letzter. We describe the $G$-strata of $\Spec R$ in
terms of certain commutative spectra. Furthermore, we show that if a
rational ideal $P$ is locally closed in $\Spec R$ then the orbit
$G.P$ is locally closed in $\Rat R$. This generalizes a standard
result on $G$-varieties. Finally, we discuss the situation where
$\GSpec R$ is a finite set.
\end{abstract}



\section*{Introduction}

\subsection{} \label{SS:setting}
This article continues our investigation \cite{mL08} of the action
of an affine algebraic group $G$ on an arbitrary associative algebra
$R$. Our focus will now be on some topological aspects of the
induced action on the set $\Spec R$ of all prime ideals of $R$, the
main themes being local closedness of $G$-orbits in $\Spec R$  and
the stratification of $\Spec R$ by means of suitable commutative
spectra. The stratification in question plays a central role in the
theory of algebraic quantum groups; see Brown and Goodearl
\cite{kBkG02} for a panoramic view of this area. Our goal here is to
develop the principal results in a context that is free of the
standard finiteness conditions, noetherianness or the Goldie
property, that underlie the pioneering works of M{\oe}glin and
Rentschler \cite{cMrR81}, \cite{cMrR84}, \cite{cMrRxx},
\cite{cMrR86} and Vonessen \cite{nVE96}, \cite{nVE98}.

Throughout, we work over an algebraically closed base field $\k$ and
we assume that the action of $G$ on $R$ is rational; the definition
will be recalled in \S\ref{SS:rationalaction}. The action will
generally be written as
\begin{equation*} \label{E:Gaction0}
G \times R \to R\,,\quad (g,r) \mapsto g.r\ .
\end{equation*}

\subsection{} \label{SS:spec}
The set $\Spec R$ carries the familiar Jacobson-Zariski topology;
see \S\ref{SS:topology} for details. Since the $G$-action on $R$
sends prime ideals to prime ideals, it induces an action of $G$ by
homeomorphisms on $\Spec R$. In the following, $G \backslash \Spec
R$ will denote the set of all $G$-orbits in $\Spec R$. We will also
consider the set $\GSpec R$ consisting of all $G$-prime ideals of
$R$. Recall that a proper $G$-stable ideal $I$ of $R$ is called
\emph{$G$-prime} if $AB \subseteq I$ for $G$-stable ideals $A$ and
$B$ of $R$ implies that $A \subseteq I$ or $B \subseteq I$.

There are surjective maps
\begin{align}
\Spec R &\stackrel{\text{can.}}{\tto} G \backslash \Spec R \ , &
P &\mapsto G.P =  \{ g.P \mid g \in G \}  \label{E:spec1} \\
\gamma \colon \Spec R &\tto \GSpec R \ , & P &\mapsto P\byG =
\bigcap_{g \in G} g.P\ . \label{E:spec}
\end{align}
See \cite[Proposition 8]{mL08} for surjectivity of $\gamma$. We will
give $G \backslash \Spec R$ and $\GSpec R$ the final topologies for
these maps: closed subsets are those whose preimage in $\Spec R$ is
closed \cite[I.2.4]{nB71}. Since \eqref{E:spec} factors through
\eqref{E:spec1}, we obtain a surjection
\begin{equation} \label{E:spec2}
G \backslash \Spec R \tto \GSpec R \ , \qquad G.P \mapsto P\byG \ .
\end{equation}
This map is continuous and closed; see \S\ref{SS:topology2}.

\subsection{} \label{SS:rat}
As in \cite{mL08}, we will be particularly concerned with the
subsets $\Rat R \subseteq \Spec R$ and $\GRat R \subseteq \GSpec R$
consisting of all rational and $G$-rational ideals of $R$,
respectively. Recall that rationality and $G$-rationality is defined
in in terms of the extended centroid $\cC(\,.\,)$ of the
corresponding factor algebra \cite{mL08}. Specifically,  $P \in
\Spec R$ is said to be \emph{rational} if $\cC(R/P) = \k$,  and $I
\in \GSpec R$ is \emph{$G$-rational} if the $G$-invariants
$\cC(R/I)^G \subseteq \cC(R/I)$ coincide with $\k$. For the
definition and basic properties of the extended centroid, the reader
is referred to \cite{mL08}. Here, we just recall that $\cC(R/P)$ and
$\cC(R/I)^G$ always are extension fields of $\k$, for any $P \in
\Spec R$ and any $I \in \GSpec R$. The extended centroid of a
semiprime noetherian (or Goldie) algebra is identical to the center
of the classical ring of quotients. In the context of enveloping
algebras of Lie algebras and related noetherian algebras, the field
$\cC(R/P)$ is commonly called the \emph{heart} (c{\oe}ur, Herz) of
the prime $P$ (e.g., \cite{pG71}, \cite{BGR73}, \cite{wBrR06}). We
will follow this tradition here.

The sets $\Rat R$ and $\GRat R$ will be viewed as topological spaces
with the topologies that are induced from $\Spec R$ and $\GSpec R$:
closed subsets of $\Rat R$ are the intersections of closed subsets
of $\Spec R$ with $\Rat R$, and similarly for $\GRat R$
\cite[I.3.1]{nB71}. The $G$-action on $\Spec R$ stabilizes $\Rat R$.
Hence we may consider the set $G \backslash \Rat R \subseteq G
\backslash \Spec R$ consisting of all $G$-orbits in $\Rat R$. We
endow $G \backslash \Rat R$ with the topology that is induced from
$G \backslash \Spec R$; this turns out to be indentical to the final
topology for the canonical surjection $\Rat R \tto G \backslash \Rat
R$ \cite[III.2.4, Prop.~10]{nB71}. By \cite[Theorem 1]{mL08}, the
surjection \eqref{E:spec2} restricts to a bijection
\begin{equation} \label{E:rat}
G \backslash \Rat R \stackrel{\text{bij.}}{\tto} \GRat R \ .
\end{equation}
This map is in fact a homeomorphism; see \S\ref{SS:ratspace}.

\subsection{}
The following diagram summarizes the various topological spaces
under consideration and their relations to each other.
%
%
\begin{equation*} \label{E:spaces}
\xymatrix{%
& & \Spec R \ar@{>>}[ddll]_{\text{can.}} \ar@{>>}[ddrr]^{\gamma} \\
& & \Rat R \ar@{>>}[ddll]|!{[dll];[drr]}\hole \ar@{>>}[ddrr]|!{[dll];[drr]}\hole \ar@{_{(}->}[u] \\
G  \backslash \Spec R \ar@{>>}[rrrr] & & & &  \GSpec R \\
G  \backslash \Rat R \ar[rrrr]_{\cong} \ar@{_{(}->}[u] & & & & \GRat
R\ar@{_{(}->}[u] }
\end{equation*}
Here, $\onto$ indicates a surjection whose target space carries the
final topology, $\into$ indicates an inclusion whose source has the
induced topology, and $\cong$ is the  homeomorphism \eqref{E:rat}.

\subsection{}
The technical core of the article is Theorem~\ref{T:fibre} which
describes the $\gamma$-fibre over a given $I \in \GSpec R$. This
fibre will be denoted by
\begin{equation*} 
\Spec_IR = \{ P \in \Spec R \mid P\byG = I \}
\end{equation*}
as in \cite{kBkG02}. The partition
\begin{equation} \label{E:Gstrat}
\Spec R = \bigsqcup_{I \in \GSpec R} \Spec_IR
\end{equation}
is called the \emph{$G$-stratification} of $\Spec R$ in
\cite[II.2]{kBkG02}. In the special case where $R$ is noetherian and
$G$ is an algebraic torus, a description of the $G$-strata
$\Spec_IR$ in terms of suitable commutative spectra was given in
\cite[II.2.13]{kBkG02}, based on work of Goodearl and Letzter
\cite{kGeL00}. For general $R$ and $G$, the intersection
\begin{equation*}
\Rat_IR = \Spec_I{R}\, \cap\, \Rat R
\end{equation*}
was treated in \cite[Theorem 22]{mL08}. Our proof of
Theorem~\ref{T:fibre}, to be given in Section~\ref{S:fibres},
elaborates on the one in \cite{mL08}.

Assuming $G$ to be connected for simplicity, we put
$$
T_I = \cC(R/I) \otimes_{\k} \k(G) \ ,
$$
where $\k(G)$ denotes the field of rational functions on $G$. The
algebra $T_I$ is a tensor product of two commutative fields and
$T_I$ has no zero divisors. The given $G$-action on $R$ and the
right regular $G$-action on $\k(G)$ naturally give rise to an action
of $G$ on $T_I$. Letting  $\Spec^G(T_I)$ denote the collection of
all $G$-stable primes of $T_I$, Theorem~\ref{T:fibre} establishes a
bijection
$$
c \colon \Spec_IR \stackrel{\text{bij.}}{\tto} \Spec^G(T_I)
$$
which is very well behaved: the map $c$ is equivariant with respect
to suitable $G$-actions, it is an order isomorphism for inclusion,
and it allows one to control hearts and rationality. For the precise
formulation of Theorem~\ref{T:fibre}, we refer to
Section~\ref{S:fibres}.

\subsection{}
Theorem~\ref{T:fibre} and the tools developed for its proof will be
used in Section~\ref{S:proof} to investigate local closedness of
rational ideals. Recall that a subset $A$ of an arbitrary
topological space $X$ is said to be \emph{locally closed} if $A$ is
closed in some neighborhood of $A$ in $X$. This is equivalent to $A$
being open in its topological closure $\bar{A}$ in $X$ or,
alternatively, $A$ being an intersection of an open and a closed
subset of $X$ \cite[I.3.3]{nB71}. A point $x \in X$ is locally
closed if $\{ x \}$ is locally closed. For $X = \Spec R$, this
amounts to the following familiar condition: a prime ideal $P$ is
locally closed in $\Spec R$ if and only if $P$ is distinct from the
intersection of all primes of $R$ that properly contain $P$. A
similar formulation holds for $X = \GSpec R$\,; see
\S\ref{SS:locclosed}. We remark that ``locally closed in $\GSpec
R$'' is referred to as ``$G$-locally closed'' in \cite{cMrRxx} and
\cite{nVE98}.

The second main result of this article is the following theorem
which will be proved in Section~\ref{S:proof}. Earlier versions
assuming additional finiteness hypotheses are due to M{\oe}glin and
Rentschler \cite[Th{\'e}or{\`e}me 3.8]{cMrR81},
\cite[Th{\'e}or{\`e}me 3]{cMrRxx} and to Vonessen \cite[Theorem
2.6]{nVE98}.

\begin{thm} \label{T:orbit}
The following are equivalent for a rational ideal $P$ of $R$:
\begin{enumerate}
\item $P$ is locally closed in $\Spec R$;
\item $\gamma(P) = P\byG$ is locally closed in $\GSpec R$.
\end{enumerate}
\end{thm}

Theorem~\ref{T:orbit} in conjunction with \eqref{E:rat} has the
following useful consequence. The corollary below extends
\cite[Cor.~2.7]{nVE98} and a standard result on $G$-varieties
\cite[Satz II.2.2]{hK84}.

\begin{cor} \label{C:orbit}
If $P \in \Rat R$  is locally closed in $\Spec R$ then the $G$-orbit
$G.P$ is open in its closure in $\Rat R$.
\end{cor}

\begin{proof}
The point $P\byG \in \GSpec R$ is locally closed by
Theorem~\ref{T:orbit}. Applying the easy fact \cite[I.3.3]{nB71}
that preimages of locally closed sets under continuous maps are
again locally closed
 to $f \colon \Rat R \into \Spec R \stackrel{\gamma}{\to} \GSpec R$,
we conclude from \eqref{E:rat} that $f^{-1}(P\byG) = G.P$ is locally
closed in $\Rat R$.
\end{proof}

\subsection{} \label{SS:DMequiv}
In order to put Theorem~\ref{T:orbit} and Corollary~\ref{C:orbit}
into perspective, we mention that rational ideals are oftentimes
locally closed in $\Spec R$. In fact, for many important classes of
algebras $R$, rational ideals are identical with the locally closed
points of $\Spec R$. Specifically, we will say that the algebra $R$
satisfies the \emph{Nullstellensatz} if the following two conditions
are satisfied:
\begin{itemize}
\item[(i)] every prime ideal of $R$ is an
intersection of primitive ideals, and
\item[(ii)] $\cen\left( \End_RV \right) = \k$ holds for every simple  $R$-module $V$.
\end{itemize}
Recall that an ideal of $R$ is said to be (right) primitive if it is
the annihilator of a simple (right) $R$-module. Hypothesis (i) is
known as the \emph{Jacobson property} while versions of (ii) are
referred to as the \emph{endomorphism property} \cite{jMcCjR87} or
the \emph{weak Nullstellensatz} \cite{dM88}, \cite{mL08}. The
Nullstellensatz is quite common. It is guaranteed to hold, for
example, if $\k$ is uncountable and the algebra $R$ is noetherian
and countably generated \cite[Corollary 9.1.8]{jMcCjR87},
\cite[II.7.16]{kBkG02}. The Nullstellensatz also holds for any
affine PI-algebra $R$ \cite[Chap.~6]{lR88}. For many other classes
of algebras satisfying the Nullstellensatz, see \cite[Chapter
9]{jMcCjR87} or \cite[II.7]{kBkG02}.

If $R$ satisfies the Nullstellensatz then the following implications
hold for all primes of $R$:
$$
\text{locally closed} \quad \Rightarrow \quad \text{primitive} \quad
\Rightarrow \quad \text{rational}.
$$
Here, the first implication is an immediate consequence of (i) while
the second follows from (ii); see \cite[Prop.~6]{mL08}. The algebra
$R$ is said to satisfy the \emph{Dixmier-M{\oe}glin equivalence} if
all three properties are equivalent for primes of $R$. Standard
examples of algebras satisfying the Dixmier-M{\oe}glin equivalence
include affine PI-algebras, whose rational ideals are in fact
maximal \cite{cP73}, and enveloping algebras of finite-dimensional
Lie algebras; see \cite[1.9]{rR87} for $\ch \k =0$. (In positive
characteristics, enveloping algebras are affine PI.) More recently,
the Dixmier-M{\oe}glin equivalence has been shown to hold for
numerous quantum groups; see \cite[II.8]{kBkG02} for an overview.

Note that the validity of the Nullstellensatz and the
Dixmier-M{\oe}glin equivalence are intrinsic to $R$. However,
$G$-actions can be useful tools in verifying the latter. Indeed,
assuming the Nullstellensatz for $R$, Theorem~\ref{T:orbit} implies
that the Dixmier-M{\oe}glin equivalence is equivalent to $P\byG$
being locally closed in $\GSpec R$ for every $P \in \Rat R$. This
condition is surely satisfied whenever $\GSpec R$ is in fact finite.

\subsection{} \label{SS:finiteintro}
The final Section~\ref{S:finiteness} briefly addresses the question
as to when $\GSpec R$ is a finite set. Besides being of interest in
connection with the Dixmier-M{\oe}glin equivalence
(\S\ref{SS:DMequiv}), this is obviously relevant for the
$G$-stratification \eqref{E:Gstrat}; see also \cite[Problem
II.10.6]{kBkG02}. Restricting ourselves to algebras $R$ satisfying
the Nullstellensatz, we show in Proposition~\ref{P:finiteness} that
finiteness of $\GSpec R$ is equivalent to the following three
conditions:
\begin{enumerate}
    \item[(i)] the ascending chain condition holds for $G$-stable semiprime ideals of $R$,
    \item[(ii)] $R$ satisfies the Dixmier-M{\oe}glin equivalence, and
    \item[(iii)] $\GRat R = \GSpec R$.
\end{enumerate}
Several versions of Proposition~\ref{P:finiteness} for noetherian
algebras $R$ can be found in \cite[II.8]{kBkG02}, where a profusion
of algebras is exhibited for which $\GSpec R$ is known to be finite.

Note that (i) above is no trouble for the standard classes of
algebras, even in the strengthened form which ignores $G$-stability.
Indeed, noetherian algebras trivially satisfy the ascending chain
condition for all semiprime ideals, and so do all affine
PI-algebras; see \cite[6.3.36']{lR88}. Moreover, as was outlined in
\S\ref{SS:DMequiv}, the Dixmier-M{\oe}glin equivalence (ii) has been
established for a wide variety of algebras. Therefore, in many
situations of interest, Proposition~\ref{P:finiteness} says in
essence that finiteness of $\GSpec R$ is tantamount to the equality
$\GRat R = \GSpec R$. This is also the only condition where the
$G$-action properly enters the picture. The article concludes with
some simple examples of torus actions satisfying (iii). Further work
is needed on how to assure the validity of (iii) under reasonably
general circumstances.

\subsection{}
This article owes a great deal to the ground breaking investigations
of
M{\oe}glin \& Rentsch\-ler
and Vonessen.
The statements of our main results as well as the basic strategies
employed in their proofs have roots in the aforementioned articles
of these authors. We have made an effort to render our presentation
reasonably self-contained while also indicating the original sources
at the appropriate points in the text. The reader interested in the
details of Sections~\ref{S:fibres} and \ref{S:proof} may wish to
have a copy of \cite{nVE98} at hand in addition to \cite{mL08}.

\begin{notat}
Our terminology and notation follows \cite{mL08}. The notations and
hypotheses introduced in the foregoing will remain in effect
throughout the paper. In particular, we will work over an
algebraically closed base field $\k$. Furthermore, $G$ will be an
affine algebraic $\k$-group and $R$ will be an associative
$\k$-algebra (with $1$) on which $G$ acts rationally by $\k$-algebra
automorphisms. For simplicity, $\otimes_\k$ will be written as
$\otimes$. Finally, for any ideal $I \trianglelefteq R$, the largest
$G$-stable ideal of $R$ that is contained in $I$ will be denoted by
$$I\byG = \bigcap_{g \in G} g.I \ .$$
\end{notat}


\section{Topological preliminaries} \label{S:topology}

\subsection{} \label{SS:topology}
Recall that the closed sets of the \emph{Jacobson-Zariski topology}
on $\Spec R$ are exactly the subsets of the form
$$
\bV(I) = \{ P \in \Spec R \mid P \supseteq I \}
$$
where $I \subseteq R$. The topological closure of a subset $A
\subseteq \Spec R$ is given by
\begin{equation} \label{E:closure}
\overline{A} = \bV(\bI(A))  \quad\text{ where}\quad \bI(A) =
\bigcap_{P \in A} P\ .
\end{equation}
The easily checked equalities $\bI(\bV(\bI(A))) = \bI(A)$ and
$\bV(\bI(\bV(I))) = \bV(I)$ show that the operators $\bV$ and $\bI$
yield inverse bijections between the collection of all closed
subsets of $\Spec R$ on one side and the collection of all semiprime
ideals of $R$ (i.e., ideals of $R$ that are intersections of prime
ideals) on the other. Thus, we have an inclusion reversing 1-1
correspondence
\begin{equation} \label{E:correspondence}
\left\{ \text{
\begin{minipage}{1.1in}
\begin{center}
closed subsets\\
$A \subseteq \Spec R$
\end{center}
\end{minipage}
} \right\}  \qquad \stackrel{1\text{-}1}{\longleftrightarrow} \qquad
\left\{ \text{\begin{minipage}{1.50in}
\begin{center}
semiprime ideals\\
$I \trianglelefteq R$
\end{center}
\end{minipage}
}
 \right\}\ .
\end{equation}
Note that the  equality $\bI(\bV(\bI(A))) = \bI(A)$ can also be
stated as
\begin{equation} \label{E:closure2}
\bI(\overline{A}) = \bI(A) \ .
\end{equation}

\subsection{}
The action of $G$ commutes with the operators $\bV$ and $\bI$:
$g.\bV(I) = \bV(g.I)$ and $g.\bI(A) = \bI(g.A)$ holds for all $g \in
G$ and all $I \subseteq R$, $A \subseteq \Spec R$. In particular,
the elements of $G$ act by homeomorphisms on $\Spec R$. Furthermore:
\begin{equation} \label{E:action1}
\text{If $g.A \subseteq A$ for a closed subset $A \subseteq \Spec R$
and $g \in G$ then $g.A = A$.}
\end{equation}
In view of the correspondence \eqref{E:correspondence}, this amounts
to saying that $g.I \supseteq I$ forces $g.I = I$ for semiprime
ideals $I \trianglelefteq R$. But this follows from the fact that
the $G$-action on $R$ is locally finite \cite[3.1]{mL08}: If $r \in
I$ satisfies $g.r \notin I$ then choose a finite-dimensional
$G$-stable subspace $V \subseteq R$ with $r \in V$ to get $I \cap V
\subsetneqq g.I \cap V \subsetneqq g^2.I \cap V \subsetneqq \dots $,
which is impossible.

Finally, consider the $G$-orbit $G.P$ of a point $P \in \Spec R$.
Since $\bI(G.P) = P\byG$, equation \eqref{E:closure} shows that the
closure of $G.P$ in $\Spec R$ is given by
$$
\overline{G.P} = \bV(P\byG) = \{ Q \in \Spec R \mid Q \supseteq
P\byG \}
$$
and \eqref{E:closure2} gives
\begin{equation*} \label{E:closure3}
\bI(\overline{G.P}) = \bI(G.P) = P\byG \ .
\end{equation*}
Thus, the correspondence \eqref{E:correspondence} restricts to an
inclusion reversing bijection
\begin{align} \label{E:correspondence2}
&\left\{ \text{ $G$-orbit closures in $\Spec R$ }
 \right\}  \qquad & &\stackrel{1\text{-}1}{\longleftrightarrow} \qquad
& &\GSpec R  \notag \\
&\qquad\qquad\ \upin & & & &\quad\upin \\
 &\qquad\qquad \overline{G.P} & &\longleftrightarrow & &\ \ \ P\byG \notag
 \ .
\end{align}

\subsection{} \label{SS:topology2}
As was mentioned in the Introduction, the spaces $G \backslash \Spec
R$ and $\GSpec R$ inherit their topology from $\Spec R$ via the
surjections in \eqref{E:spec1} and \eqref{E:spec}:
\begin{equation*} \label{E:specspaces}
\xymatrix{%
& \Spec R \ar@{>>}[dl]_{\pi = \text{can.}} \ar@{>>}[dr]^{\gamma} \\
G  \backslash \Spec R   & &  \GSpec R }
\end{equation*}
with $\pi(P) = G.P$ and $\gamma(P) = P\byG$. Closed sets in $G
\backslash \Spec R$ and in $\GSpec R$ are exactly those subsets
whose preimage in $\Spec R$ is closed. In both cases, preimages are
also $G$-stable, and hence they have the form $\bV(I)$ for some
$G$-stable semiprime ideal $I \trianglelefteq R$.

Let $C$ be a closed subset of $\GSpec R$ and write $\gamma^{-1}(C) =
\bV(I)$ as above. Since $P \supseteq I$ is equivalent to $P\byG
\supseteq I$ for $P \in \Spec R$, we obtain
\begin{equation} \label{E:Gclosure}
\begin{split}
C = \gamma(\bV(I)) &= \{ P\byG \mid P \in \Spec R, P\supseteq I \} \\
&= \{ P\byG \mid P \in \Spec R, P\byG \supseteq I \} \\
&= \{ J \in \GSpec R \mid J \supseteq I \}\ .
\end{split}
\end{equation}
Conversely, if $C = \{ J \in \GSpec R \mid J \supseteq I \}$ for
some $G$-stable semiprime ideal $I \trianglelefteq R$ then
$\gamma^{-1}(C) = \bV(I)$ is closed in $\Spec R$ and so $C$ is
closed in $\GSpec R$. Thus, the closed subsets of $\GSpec R$ are
exactly those of the form \eqref{E:Gclosure}. The partial order on
$\GSpec R$ that is given by inclusion can be expressed in terms of
orbit closures by \eqref{E:correspondence2}: for $P,Q \in \Spec R$,
we have
\begin{equation} \label{E:order}
Q\byG \supseteq P\byG \iff \bar{G.Q} \subseteq \bar{G.P} \iff Q \in
\bar{G.P} \ .
\end{equation}

The map $\gamma$ factors through $\pi$; so we have a map
\begin{equation*} \label{E:gamma'}
\bga \colon G  \backslash \Spec R \to  \GSpec R
\end{equation*}
with $\gamma = \bga \circ \pi$ as in \eqref{E:spec2}. Since
$\gamma^{-1}(C) = \pi^{-1}(\bga^{-1}(C))$ holds for any $C \subseteq
\GSpec R$, the map $\bga$ is certainly continuous. Moreover, if $B
\subseteq G \backslash \Spec R$ is closed then $A = \pi^{-1}(B)
\subseteq \Spec R$ satisfies $\gamma(A) = \bga(B)$ and $A = \bV(I)$
for some $G$-stable semiprime ideal $I \trianglelefteq R$. Hence,
$\bga(B) = \gamma(\bV(I))$ is closed in $\GSpec R$ by
\eqref{E:Gclosure}. This shows that $\bga$ is a closed map.

\subsection{} \label{SS:locclosed}
Recall that a subset $A$ of an arbitrary topological space $X$ is
\emph{locally closed} if and only if $\bar{A} \setminus A = \bar{A}
\cap A^\complement$ is a closed subset of $X$. By \eqref{E:closure},
a prime ideal $P$ of $R$ is locally closed in $\Spec R$ if and only
if $\bV(P) \setminus \{ P \} = \{ Q \in \Spec R \mid Q \supsetneqq P
\}$ is a closed subset of $\Spec R$, which in turn is equivalent to
the condition
\begin{equation} \label{E:locclosed}
P \subsetneqq  
\bI(\bV(P) \setminus \{ P \}) = \bigcap_{\substack{Q \in \Spec R \\
Q \supsetneqq P}} Q  \ .
\end{equation}
Similarly, \eqref{E:Gclosure} implies that a $G$-prime ideal $I$ of
$R$ is locally closed in $\GSpec R$ if and only if
\begin{equation} \label{E:Glocclosed}
I \subsetneqq  
\bigcap_{\substack{J \in \GSpec R \\
J \supsetneqq I}} J  \ .
\end{equation}

\begin{lem} \label{L:finiteindex}
Let $P \in \Spec R$ and let $N$ be a normal subgroup of $G$ having
finite index in $G$. Then $P\byG$ is locally closed in $\GSpec R$ if
and only if $P\byN$ is locally closed in $\NSpec R$.
\end{lem}

\begin{proof}
For brevity, put $X = \NSpec R$, $Y = \GSpec R$ and $H = G/N$. Thus,
$H$ is a finite group acting by homeomorphisms on $X$. From
\eqref{E:spec2} we obtain a continuous surjection $X \to Y$, $I
\mapsto I\byH$, whose fibres are easily seen to be the $H$-orbits in
$X$. Thus, we obtain a continuous bijection $H \backslash X \to Y$.
From the description \eqref{E:Gclosure} of the closed sets in $X$
and $Y$, we further see that this bijection is closed, and hence it
is a homeomorphism. Therefore, the image $I\byH$ of a given $I \in
X$ is locally closed in $Y$ if and only if the orbit $H.I$ is
locally closed in $X$; see \cite[I.5.3, Cor.~of Prop.~7]{nB71}.
Finally, with $\bar{\phantom{.P}}$ denoting topological closure in
$X$, one easily checks that
$$
\bar{H.I} \setminus H.I = \bigcup_{h \in H} h.\left( \bar{\{ I \}}
\setminus \{ I \} \right)
$$
and, consequently, $\left( \bar{H.I} \setminus H.I \right) \cap
\bar{\{ I \}} = \bar{\{ I \}} \setminus \{ I \}$. Thus, $H.I$ is
locally closed if and only if $I$ is locally closed, which proves
the lemma.
\end{proof}

\subsection{} \label{SS:ratspace}
We now turn to the space $\Rat R$ of rational ideals of $R$ and the
associated spaces $G \backslash \Rat R$ and $\GRat R$ with the
induced topologies from $\Spec R$, $G \backslash \Spec R$ and
$\GSpec R$, respectively; see \S\ref{SS:rat}. Restricting the maps
$\pi = \text{can.}$ and $\bga$ in \S\ref{SS:topology2}, we obtain a
commutative triangle of continuous maps
\begin{equation*} \label{E:ratspaces}
\xymatrix{%
& \Rat R \ar@{>>}[dl]_{\pir=\text{can.}} \ar@{>>}[dr]^{\gar} \\
G  \backslash \Rat R \ar[rr]^{\bgr}  & &  \GRat R }
\end{equation*}
The map $\bgr$, identical with \eqref{E:rat}, is bijective.
Furthermore, if $B \subseteq G \backslash \Rat R$ is closed then, as
in \S\ref{SS:topology2}, we have $\bgr(B) = \{ J \in \GRat R \mid J
\supseteq I \}$ for some $G$-stable semiprime ideal $I
\trianglelefteq R$; so $\bgr(B)$ is closed in $\GRat R$. This shows
that $\bgr$ is a homeomorphism. For general reasons, the quotient
map $\pir$ is open and the topology on $G  \backslash \Rat R$ is
identical to the final topology for $\pir$ \cite[III.2.4, Lemme 2
and Prop.~10]{nB71}. By virtue of the homeomorphism $\bgr$, the same
holds for $\GRat R$ and $\gar$.

We remark that injectivity of $\bgr$ and \eqref{E:order} imply that
\begin{equation*} 
Q\byG \supsetneqq P\byG \iff Q \in \bar{G.P} \setminus G.P
\end{equation*}
holds for $P,Q \in \Rat R$. Here, $\bar{\phantom{.P}}$ can be taken
to be the closure in $\Spec R$ or in $\Rat R$. Focusing on the
latter interpretation, we easily conclude that
\begin{equation*}
\text{$P\byG$ is locally closed in $\GRat R$} \iff \text{$G.P$ is
open in its closure in $\Rat R$.}
\end{equation*}
Alternatively, in view of the homeomorphism $\bgr$, local closedness
of $P\byG \in \GRat R$ is equivalent to local closedness of the
point $G.P \in G \backslash \Rat R$, which in turn is equivalent to
local closedness of the preimage $G.P = \pir^{-1}(G.P) \subseteq
\Rat R$; see \cite[I.5.3, Cor.~of Prop.~7]{nB71}.


\section{Ring theoretical preliminaries} \label{S:rings}

\subsection{} \label{SS:C}

The extended centroid of a ring $U$ will be denoted by
$$
\cC(U)\ .
$$
By definition, $\cC(U)$ is the center of the Amitsur-Martindale ring
of quotients of $U$. We briefly recall some basic definitions and
facts. For details, the reader is referred to \cite{mL08}.

The ring $U$ is said to be \emph{centrally closed} if $\cC(U)
\subseteq U$. If $U$ is semiprime then
$$
\til{U} = U\cC(U)
$$
is a centrally closed semiprime subring of the Amitsur-Martindale
ring of quotients of $U$; it is called the \emph{central closure} of
$U$. Furthermore, if $U$ is prime then $\til{U}$ is prime as well
and $\cC(U)$ is a field. Consequently, for any $P \in \Spec U$, we
have a commutative field $\cC(U/P)$.

If $\Gamma$ is any group acting by automorphisms on $U$ then the
action of $\Gamma$ extends uniquely to an action on the
Amitsur-Martindale ring of quotients of $U$, and hence $\Gamma$ acts
on $\cC(U)$; see \cite[2.3]{mL08}. If $I$ is a $\Gamma$-prime ideal
of $U$ then the ring of $\Gamma$-invariants $\cC(U/I)^\Gamma$ is a
field; see \cite[Prop.~9]{mL08}.

\subsection{} \label{SS:extension}

A ring homomorphism $\phi \colon U \to V$ is called
\emph{centralizing} if the ring $V$ is generated by the image
$\phi(U)$ and its centralizer, $C_V(\phi(U)) = \{ v \in V \mid v
\phi(u) = \phi(u)v \ \forall u \in U \}$; see \cite[1.5]{mL08}. The
lemma below is a special case of \cite[Lemma 4]{mL08} and it can
also be found in an earlier unpublished preprint of George Bergman
\cite[Lemma 1]{gBxx}.

\begin{lem} \label{L:extension}
Let $\phi \colon U \into V$ be a centralizing embedding of prime
rings. Then $\phi$ extends uniquely to a homomorphism $\til{\phi}
\colon \til{U} \to \til{V}$ between the central closures of $U$ and
$V$. The extension $\til{\phi}$ is again injective and centralizing.
In particular, $\til{\phi}(\cC(U)) \subseteq \cC(V)$.
\end{lem}

\begin{proof}
If $I$ is a nonzero ideal of $U$ then $\phi(I)V = V\phi(I)$ is a
nonzero ideal of $V$. Hence $\phi(I)$ has zero left and right
annihilator in $V$. The existence of the desired extension
$\til{\phi}$ now follows from \cite[Lemma 4]{mL08}, since $\cC_\phi
= \cC(U)$ holds in the notation of that result. Uniqueness of
$\til{\phi}$ as well as injectivity and the centralizing property
are immediate from \cite[Prop.~2(ii)(iii)]{mL08}. For example, in
order to guarantee that $\til{\phi}$ is centralizing, it suffices to
check that $C_V(\phi(U))$ centralizes $\til{\phi}(\cC(U))$. To prove
this, let $q \in \cC(U)$ and $v \in C_V(\phi(U))$ be given. By
\cite[Prop.~2(ii)]{mL08} there is a nonzero ideal $I \trianglelefteq
U$ so that $qI \subseteq U$. For $u \in I$, one computes
$$
\til{\phi}(q)v\phi(u) = \til{\phi}(q)\phi(u)v = \phi(qu)v =
v\phi(qu) = v\til{\phi}(q)\phi(u)\ .
$$
This shows that $[v,\til{\phi}(q)]\phi(I) = 0$ and
\cite[Prop.~2(iii)]{mL08} further implies that $[v,\til{\phi}(q)] =
0$.
\end{proof}

The lemma implies in particular that every automorphism of a prime
ring $U$ extends uniquely to the central closure $\til{U}$. (A more
general fact was already mentioned above.) We will generally use the
same notation for the extended automorphism of $\til{U}$.

\subsection{} \label{SS:bijection}

The essence of the next result goes back to Martindale et al.
\cite{wM69}, \cite{EMO75}.

\begin{prop} \label{L:bijection}
Let $U$ be a centrally closed prime ring and let $V$ be any
$C$-algebra, where $C = \cC(U)$. Then there are bijections
\begin{align*}
\{ P \in \Spec(U\otimes_CV) &\mid P \cap U = 0\} &
&\stackrel{1\text{-}1}{\longleftrightarrow}
& &\Spec V   \\
 &P & &\longmapsto & &P\cap V \\
 U&\otimes_C\fp & &\longmapsfrom & &\quad \fp
 \ .
\end{align*}
These bijections are inverse to each other and they are equivariant
with respect to all automorphisms of $U\otimes_CV$ that stabilize
both $U$ and $V$. Furthermore, hearts are preserved:
$$
\cC((U \otimes_C V)/P) \cong \cC(V/P \cap V) \ .
$$
\end{prop}

\begin{proof}
The extension $V \into U\otimes_CV$ is centralizing. Therefore,
contraction $P \mapsto P\cap V$ is a well-defined map
$\Spec(U\otimes_CV) \to \Spec V$ which clearly has the stated
equivariance property.

If $V$ is a prime ring then so is $U\otimes_CV$; this follows from
the fact that every nonzero ideal of $U\otimes_CV$ contains a
nonzero element of the form $u \otimes v$ with $u \in U$ and $v \in
V$; see \cite[Lemma 3(a)]{mL08} or \cite[Theorem 3.8(1)]{EMO75}. By
\cite[Cor.~2.5]{jM82} we also know that $\cC(U\otimes_CV) = \cC(V)$.
Consequently, $\fp \mapsto U\otimes_C\fp = (U \otimes_C V)\fp$ gives
a map $\Spec V \to \{ P \in \Spec(U\otimes_CV) \mid P \cap U = 0\}$
which preserves hearts.

Finally, by \cite[Lemma 3(c)]{mL08}, the above maps are inverse to
each other.
\end{proof}

\subsection{} \label{SS:bijection2}

Let $U$ be a prime ring. If $U$ is an algebra over some commutative
field $F$ then the central closure $\til{U}$ is an $F$-algebra as
well, because $\cen(U)\subseteq \cC(U) = \cen(\til{U})$.

\begin{prop} \label{L:bijection2}
Let $U$ and $V$ be algebras over some commutative field $F$, with
$U$ prime. Then there is a bijection
\begin{align*}
\{ \til{P} \in \Spec(\til{U} \otimes_F V) &\mid \til{P} \cap \til{U}
 = 0 \}&  &\tto & \{ P \in \Spec&(U \otimes_F V) \mid P \cap U  = 0 \} \\
&\til{P} & &\longmapsto & &\til{P} \cap (U \otimes_F V) \ .
\end{align*}
This bijection and its inverse are inclusion preserving and
equivariant with respect to all automorphisms of $\til{U} \otimes_F
V$ that stabilize both $\til{U}$ and $U \otimes_F V$. Furthermore,
hearts are preserved under this bijection.
\end{prop}

\begin{proof}
Since the extension $U \otimes_F V \into \til{U} \otimes_F V$ is
centralizing, the contraction map $\til{P} \mapsto \til{P} \cap (U
\otimes_F V)$ sends primes to primes, and hence it yields a
well-defined map between the sets in the proposition. This map is
clearly inclusion preserving and equivariant as stated.

For surjectivity, let $P$ be a prime of $U \otimes_F V$ such that $P
\cap U = 0$. The canonical map
$$
\phi \colon U \into U \otimes_F V \onto W = (U \otimes_F V)/P
$$
is a centralizing embedding of prime rings. By
Lemma~\ref{L:extension}, there is a unique extension to central
closures,
$$
\til{\phi} \colon \til{U} \tto \til{W}\ .
$$
The image of the canonical map $\psi \colon V \into U \otimes_F V
\onto W$ centralizes $\phi(U)$, and hence it also centralizes
$\til{\phi}(\til{U})$; see the proof of Lemma~\ref{L:extension}.
Therefore, $\til{\phi}$ and $\psi$ yield a ring homomorphism
$$
\til{\phi}_V \colon \til{U} \otimes_F V \tto \til{W}\ .
$$
Put $\til{P} = \Ker \til{\phi}_V$. Since $\til{\phi}_V$ extends the
canonical map $U \otimes_F V \onto W = (U \otimes_F V)/P$, we have
$\til{P} \cap (U \otimes_F V) = P$ and
\begin{equation} \label{E:W}
W \subseteq \Im \til{\phi}_V = (\til{U} \otimes_F V)/\til{P}
\subseteq \til{W} \ .
\end{equation}
In particular, every nonzero ideal of $\Im \til{\phi}_V$ has a
nonzero intersection with $W$, and hence $\Im \til{\phi}_V$ is a
prime ring and $\til{P}$ is a prime ideal. Since $\til{P} \cap U = P
\cap U = 0$, it follows that $\til{P} \cap \til{U} = 0$. This proves
surjectivity. Furthermore, since the inclusions in \eqref{E:W} are
centralizing inclusions of prime rings, they yield inclusions of
extended centroids, $\cC(W) \subseteq \cC((\til{U} \otimes_F
V)/\til{P}) \subseteq \cC(\til{W}) = \cC(W)$ by
Lemma~\ref{L:extension}. Therefore, $\til{P}$ has the same heart as
$P$.

To prove injectivity, let $\til{P}$ and $\til{P}'$ be primes of
$\til{U} \otimes_F V$, both disjoint from $\til{U} \setminus \{0\}$,
such that $\til{P} \cap (U \otimes_F V) \subseteq \til{P}' \cap (U
\otimes_F V)$. We claim that $\til{P} \subseteq \til{P}'$. Indeed,
it follows from \cite[Prop.~2(ii)]{mL08} that, for any $q \in
\til{U} \otimes_F V$, there is a nonzero ideal $I$ of $U$ such that
$qI \subseteq U \otimes_F V$. For $q \in \til{P}$, we conclude that
$qI \subseteq \til{P}'$. Since $I(\til{U}\otimes_F V)$ is an ideal
of $\til{U}\otimes_F V$ that is not contained in $\til{P}'$, we must
have $q \in \til{P}'$. Therefore, $\til{P} \subseteq \til{P}'$ as
claimed. Injectivity follows and we also obtain that the inverse
bijection preserves inclusions. This completes the proof.
\end{proof}

We will apply the equivariance property of the above bijection to
automorphisms of the form $\alpha \otimes_F \beta$ with $\alpha \in
\Aut_{F\mbox{-}\text{alg}}(U)$, extended uniquely to $\til{U}$ as in
\S\ref{SS:extension}, and $\beta \in \Aut_{F\mbox{-}\text{alg}}(V)$.

\subsection{} \label{SS:VE}

The following two technical results have been extracted from
M{\oe}glin and Rentsch\-ler \cite[3.4-3.6]{cMrR81}; see also
Vonessen \cite[proof of Prop.~8.12]{nVE98}. As above, $F$ denotes a
commutative field. If a group $\Gamma$ acts on a ring $U$ then $U$
is called a $\Gamma$-ring; similarly for fields. A $\Gamma$-ring $U$
is called $\Gamma$-simple if $U$ has no $\Gamma$-stable ideals other
than $0$ and $U$.

\begin{lem} \label{L:fields}
Let $F \subseteq L \subseteq K$ be $\Gamma$-fields. Assume that $K =
\Fract A$ for some $\Gamma$-stable $F$-subalgebra $A$ which is
$\Gamma$-simple and affine (finitely generated). Then $L = \Fract B$
for some $\Gamma$-stable $F$-subalgebra $B$ which is $\Gamma$-simple
and generated by finitely many $\Gamma$-orbits.
\end{lem}

\begin{proof}
We need to construct a $\Gamma$-stable $F$-subalgebra $B \subseteq
L$ satisfying
\begin{itemize}
\item[(i)] $L = \Fract B$,
\item[(ii)] $B$ is generated as $F$-algebra by finitely many
$\Gamma$-orbits, and
\item[(iii)] $B$ is $\Gamma$-simple.
\end{itemize}
Note that $L$ is a finitely generated field extension of $F$,
because $K$ is. Fix a finite set $X_0 \subseteq L$ of field
generators and let $B_0 \subseteq L$ denote the $F$-subalgebra that
is generated by $\bigcup_{x \in X_0} \Gamma.x$. Then $B_0$ certainly
satisfies (i) and (ii). We will show that the intersection of all
nonzero $\Gamma$-stable semiprime ideals of $B_0$ is nonzero.
Consider the algebra $B' = B_0A \subseteq K$; this algebra is
$\Gamma$-stable and affine over $B_0$. By generic flatness
\cite[2.6.3]{jD96}, there exists some nonzero $t \in B_0$ so that
$B'[t^{-1}]$ is free over $B_0[t^{-1}]$. We claim that if $\fb_0$ is
any $\Gamma$-stable semiprime ideal of $B_0$ not containing $t$ then
$\fb_0$ must be zero. Indeed, $\fb_0 B_0[t^{-1}]$ is a proper ideal
of $B_0[t^{-1}]$, and hence $\fb_0 B'[t^{-1}]$ is a proper ideal of
$B'[t^{-1}]$. Therefore, $\fb_0 A \cap A$ is a proper ideal of $A$
which is clearly $\Gamma$-stable. Since $A$ is $\Gamma$-simple, we
must have $\fb_0 A \cap A = 0$. Finally, since $\fb_0 \subseteq K =
\Fract A$, we conclude that $\fb_0 = 0$ as desired. Thus, $t$
belongs to every nonzero $\Gamma$-stable semiprime ideal of $B_0$.
Now let $B$ be the $F$-subalgebra of $L$ that is generated by $B_0$
and the $\Gamma$-orbit of $t^{-1}$. Clearly, $B$ still satisfies (i)
and (ii). Moreover, if $\fb$ is any nonzero $\Gamma$-stable
semiprime ideal of $B$ then $\fb \cap B_0$ is a $\Gamma$-stable
semiprime ideal of $B_0$ which is nonzero, because $L = \Fract B_0$.
Hence $t \in \fb$ and so $\fb = B$. This implies that $B$ is
$\Gamma$-simple which completes the construction of $B$.
\end{proof}

The lemma above will only be used in the proof of the following
``lying over'' result which will be crucial later on. For a given
$\Gamma$-ring $U$, we let
$$
\Spec^{\Gamma}U
$$
denote the collection of all $\Gamma$-stable primes of $U$. These
primes are certainly $\Gamma$-prime, but the converse need not hold
in general.

\begin{prop} \label{L:VE}
Let $U$ be a prime $F$-algebra and let $\Gamma$ be a group acting by
$F$-algebra automorphisms on $U$. Suppose that there is a
$\Gamma$-equivariant embedding $\cC(U) \into K$, where $K$ is a
$\Gamma$-field such that $K = \Fract A$ for some $\Gamma$-stable
affine $F$-subalgebra $A$ which is $\Gamma$-simple.

Then there exists a nonzero ideal $D \trianglelefteq U$ such that,
for every $P \in \Spec^{\Gamma}U$ not containing $D$, there is a
$\til{P} \in \Spec^{\Gamma}\til{U}$ satisfying $\til{P} \cap U = P$.
\end{prop}

\begin{proof}
Applying Lemma~\ref{L:fields} to the given embedding $F \subseteq C
= \cC(U) \into K$ we obtain a $\Gamma$-stable $F$-subalgebra $B
\subseteq C$ such that $C = \Fract B$, $B$ is $\Gamma$-simple and
$B$ is generated as $F$-algebra by finitely many $\Gamma$-orbits.
Fix a finite subset $X \subseteq B$ such that $B$ is generated by
$\bigcup_{x \in X} G.x$ and let
$$
D = \{ u \in U \mid xu \in U \text{ for all $x \in X$}\}\ ;
$$
this is a nonzero ideal of $U$ by \cite[Prop.~2(ii)]{mL08}. In order
to show that $D$ has the desired property, we first make the
following
\begin{claim}
Suppose $P \in \Spec^{\Gamma}U$ does not contain $D$. Then, given
$w_1,\dots,w_n \in UB \subseteq \til{U}$, there exists an ideal $I
\trianglelefteq U$ with $I \nsubseteq P$ and such that $w_iI
\subseteq U$ for all $i$.
\end{claim}
To see this, note that every element of $UB$ is a finite sum of
terms of the form
$$
w = u(g_1.x_1) \cdots (g_r.x_r)
$$
with $u \in U$, $g_j \in G$ and $x_j \in X$. The ideal $J = \left(
\bigcap_{j=1}^r g_j.D \right)^r$ of $U$ satisfies
$$
wJ \subseteq ug_1.(x_1D) \cdots g_r.(x_rD) \subseteq U \ .
$$
Moreover, $J \nsubseteq P$, because otherwise $g_j.D \subseteq P$
for some $j$ and so $D \subseteq P$ contrary to our hypothesis. Now
write the given $w_i$ as above, collect all occuring $g_j$ in the
(finite) subset $E \subseteq G$ and let $s$ be the largest occuring
$r$. Then the ideal $I = \left( \bigcap_{g \in E} g.D \right)^s$
does what is required.

Next, we show that
$$
PB \cap U = P \ .
$$
Indeed, for any $u \in PB \cap U$, the above claim yields an ideal
$I \trianglelefteq U$ with $I \nsubseteq P$ and such that $uI
\subseteq P$. Since $P$ is prime, we must have $u \in P$ which
proves the above equality. Now choose an ideal $Q \trianglelefteq
UB$ which contains the ideal $PB$ and is maximal with respect to the
condition $Q \cap U = P$ (Zorn's Lemma). It is routine to check that
$Q$ is prime. Therefore, $\til{Q} = Q\byGa$ is at least a
$\Gamma$-prime ideal of $UB$ satisfying $\til{Q} \cap U = P$. We
show that $\til{Q}$ is in fact prime. Let $w_1, w_2 \in UB$ be given
such that $w_1UBw_2 \subseteq \til{Q}$. Choosing $I$ as in the claim
for $w_1,w_2$, we have $w_1IUw_2I \subseteq U \cap \til{Q} = P$.
Since $P$ is prime, we conclude that $w_jI \subseteq P$ for $j=1$ or
$2$. Hence, $w_jIB \subseteq \til{Q} = \bigcap_{g \in \Gamma} g.Q$.
Since each $g.Q$ is prime in $UB$ and $IB$ is an ideal of $UB$ not
contained in $g.Q$, we obtain that $w_j \in \bigcap_{g \in \Gamma}
g.Q = \til{Q}$. This shows that $\til{Q}$ is indeed prime.

Finally, $\til{U}$ is the (central) localization of $UB$ at the
nonzero elements of $B$. Furthermore, $\til{Q} \cap B = 0$, since
$\til{Q} \cap B$ is a proper $\Gamma$-stable ideal of $B$.  It
follows that $\til{P} = \til{Q}C$ is a prime ideal of $\til{U}$
which is clearly $\Gamma$-stable and satisfies $\til{P} \cap UB =
\til{Q}$. Consequently, $\til{P} \cap U = \til{Q} \cap U = P$,
thereby completing the proof.
\end{proof}


\section{Description of $G$-strata} \label{S:fibres}

\subsection{} \label{SS:rationalaction}
We now return to the setting of \S\ref{SS:setting}. The $G$-action
on $R$ will be denoted by
\begin{equation} \label{E:Gaction}
\rho = \rho_R \colon G \tto \Aut_{\k\mbox{-}\text{alg}}(R)
\end{equation}
when it needs to be explicitly referred to; so
\begin{equation*} \label{E:Gaction1}
g.r = \rho(g)(r) \,.
\end{equation*}
Recall from \cite[3.1, 3.4]{mL08} that rationality of the action of
$G$ on $R$ is equivalent to the existence of a $\k$-algebra map
\begin{equation*} 
\Delta_R \colon R \tto  R \otimes \k[G]\ , \quad r \mapsto \sum_i
r_i \otimes f_i
\end{equation*}
such that
\begin{equation*} \label{E:action'}
g.r  = \sum_i r_i f_i(g)
\end{equation*}
holds for all $g \in G$ and $r \in R$. Here, $\k[G]$ denotes the
Hopf algebra of regular functions on $G$, as usual. The
$\k[G]$-linear extension of the map $\Delta_R$ is an automorphism of
$\k[G]$-algebras which will also be denoted by $\Delta_R$:
\begin{equation} \label{E:Delta}
\Delta_R \colon R \otimes \k[G] \iso  R \otimes \k[G] \ ;
\end{equation}
see \cite[3.4]{mL08}.

\subsection{} \label{SS:intertwining}

As in \cite{mL08}, the right and left \emph{regular representations}
of $G$ will be denoted by
$$
\rho_r, \rho_\ell \colon G \to \Aut_{\k\mbox{-}\text{alg}}(\k[G])\ ;
$$
they are defined by $\left(\rho_r(x)f\right)(y) = f(yx)$ and
$\left(\rho_\ell(x)f\right)(y) = f(x^{-1}y)$ for $x,y \in G$.

The group $G$ acts (rationally) on the $\k$-algebra $R \otimes
\k[G]$ by means of the maps $\Id_R \otimes \rho_{r,\ell}$ and $\rho
\otimes \rho_{r,\ell}$. The following intertwining formulas hold for
all $g \in G$:
\begin{equation} \label{E:reg1}
\Delta_R \circ \left( \rho \otimes \rho_r \right)(g)  = \left( \Id_R
\otimes \rho_r \right)(g) \circ \Delta_R
\end{equation}
and
\begin{equation} \label{E:reg2}
\Delta_R \circ \left( \Id_R \otimes \rho_\ell \right)(g) = \left(
\rho \otimes \rho_\ell \right)(g) \circ \Delta_R \ ,
\end{equation}
where $\Delta_R$ is the automorphism \eqref{E:Delta}; see
\cite[3.3]{mL08}.

In the following,
$$
\k(G) = \Fract \k[G]
$$
will denote the algebra $\k(G)$ of rational functions on $G$; this
is the full ring of fractions of $\k[G]$, and $\k(G)$ is also equal
to the direct product of the rational function fields of the
irreducible components of $G$ \cite[AG 8.1]{aB91}. The above
$G$-actions extend uniquely to $\k(G)$ and to $R \otimes \k(G)$. We
will use the same notations as above for the extended actions. The
intertwining formulas \eqref{E:reg1} and \eqref{E:reg2} remain valid
in this setting, with $\Delta_R$ replaced by its unique extension to
a $\k(G)$-algebra automorphism
\begin{equation} \label{E:Delta2}
\Delta_R \colon R \otimes \k(G) \iso R \otimes \k(G) \ .
\end{equation}

\subsection{} \label{SS:fibre}

We are now ready to describe the $G$-stratum
\begin{equation*}
\Spec_IR =  \{ P \in \Spec R \mid P\byG = I \}
\end{equation*}
over a given $I \in \GSpec R$. For simplicity, we will assume that
$G$ is connected; so $\k(G)$ is a field which is unirational over
$\k$ \cite[18.2]{aB91}. Furthermore, $I$ is a prime ideal of $R$ by
\cite[Prop.~19(a)]{mL08} and so $\cC(R/I)$ is a field as well. The
group $G$ acts on $R/I$ by means of the map $\rho$ in
\eqref{E:Gaction}. This action extends uniquely to an action on the
central closure $\til{R/I}$, and hence we also have a $G$-action on
$\cC(R/I) = \cen(\til{R/I})$. Denoting the latter two actions by
$\rho$ again, we obtain $G$-actions $\rho \otimes \rho_r$ on
$\til{R/I} \otimes \k(G)$ and on $\cC(R/I) \otimes \k(G)$. As in
\S\ref{SS:VE}, we will write
\begin{equation*} 
\Spec^G(T) =  \{ \fp \in \Spec T \mid \text{$\fp$ is $(\rho \otimes
\rho_r)(G)$-stable} \}
\end{equation*}
for $T = \cC(R/I) \otimes \k(G)$ or $T = \til{R/I} \otimes \k(G)$.
The group $G$ also acts on both algebras $T$ via $\Id \otimes
\rho_\ell$ and the latter action commutes with $\rho \otimes
\rho_r$. Hence, $G$ acts on $\Spec^G(T)$ through $\Id \otimes
\rho_\ell$. We are primarily interested in the first of these
algebras,
$$
T_I = \cC(R/I) \otimes \k(G)\ ,
$$
a commutative domain and a tensor product of two fields.

\begin{thm} \label{T:fibre}
For a given $I \in \GSpec R$, let $T_I = \cC(R/I) \otimes \k(G)$.
There is a bijection
$$
c \colon \Spec_IR \stackrel{\text{\rm bij.}}{\tto} \Spec^G(T_I)
$$
having the following properties, for $P,P' \in \Spec_IR$ and $g \in
G$:
\begin{enumerate}
\item \textnormal{$G$-equivariance:} $c(g.P) = (\Id \otimes
\rho_\ell(g))(c(P))$\,;
\item \textnormal{inclusions:} $P \subseteq P' \iff c(P) \subseteq
c(P')$\,;
\item \textnormal{hearts:} there is an
isomorphism of \,$\k(G)$-fields $\Psi_P \colon \cC(T_I/c(P)) \iso
\cC\left( (R/P) \otimes \k(G) \right)$ satisfying
\begin{align*}
\Psi_P \circ (\rho \otimes \rho_r)(g) &= (\Id_{R/P} \otimes
\rho_r)(g)\circ \Psi_P \,, \\
\Psi_{g.P} \circ (\Id \otimes \rho_\ell)(g) &= (\rho \otimes
\rho_\ell)(g)\circ \Psi_P \,;
\end{align*}
\item \textnormal{rationality:} $P$ is rational if and
only if $T_I/c(P) \cong \k(G)$.
\end{enumerate}
\end{thm}

Note that $\cC(T_I/c(P))$ in (c) is just the classical field of
fractions of the commutative domain $T_I/c(P)$. Furthermore, in the
second identity in (c), we have
\begin{align*}
(\Id \otimes \rho_\ell)(g) \colon \cC(T_I/c(P)) &\iso
\cC(T_I/c(g.P)) \\
(\rho \otimes \rho_\ell)(g) \colon \cC\left( (R/P) \otimes \k(G)
\right) &\iso \cC\left( (R/g.P) \otimes \k(G) \right)
\end{align*}
in the obvious way. For (d), recall from \eqref{E:rat} that there
exists a rational $P \in \Spec_IR$ if and only if $I$ is
$G$-rational.

\begin{proof}
Replacing $R$ by $R/I$, we may assume that $I=0$. Our goal is to
establish a bijection between $\Spec_0R = \{ P \in \Spec R \mid
P\byG = 0 \}$ and $\Spec^G(T_0)$. For (a), this bijection needs to
be equivariant for the $G$-action by $\rho$ on $R$ and by $\Id
\otimes \rho_\ell$ on $T_0$.

As was pointed out above, $R$ is a prime ring. For brevity, we will
put
$$
C = \cC(R)\quad \text{and} \quad K=\k(G) \ .
$$
Thus, $C$ and $K$ are fields and $T_0 = C \otimes K$. Let $\til{R} =
RC$ denote the central closure of $R$; so $\til{R}$ is a centrally
closed prime ring and
$$
\til{R} \otimes K = \til{R} \otimes_{C} T_0 \ .
$$
By Proposition~\ref{L:bijection}, $\Spec T_0$ is in bijection with
the set of all primes $\til{Q} \in \Spec(\til{R} \otimes K)$ such
that $\til{Q} \cap \til{R} = 0$. This bijection is equivariant with
respect to the subgroups $(\rho \otimes \rho_r)(G)$ and $(\Id
\otimes \rho_\ell)(G)$ of $\Aut_{\k\mbox{-}\text{alg}}(\til{R}
\otimes K)$, because these subgroups stabilize both $\til{R}$ and
$T_0$. Therefore, the bijection in Proposition~\ref{L:bijection}
restricts to a bijection
$$
\Spec^G(T_0) \stackrel{1\text{-}1}{\longleftrightarrow} \{\til{Q}
\in \Spec^G(\til{R} \otimes K) \mid \til{Q} \cap \til{R} = 0 \}
$$
which is equivariant for the $G$-action by $\Id \otimes \rho_\ell$
on $T_0$ and on $\til{R} \otimes K$. Furthermore,
Proposition~\ref{L:bijection2} gives a bijection $\{\til{Q} \in
\Spec(\til{R} \otimes K) \mid \til{Q} \cap \til{R} = 0 \}
\stackrel{1\text{-}1}{\longleftrightarrow} \{ Q \in \Spec(R \otimes
K) \mid Q \cap R = 0 \}$ and this bijection is equivariant with
respect to both $(\rho \otimes \rho_r)(G)$ and $(\Id \otimes
\rho_\ell)(G)$. As above, we obtain a bijection
$$
\{\til{Q} \in \Spec^G(\til{R} \otimes K) \mid \til{Q} \cap \til{R} =
0 \} \stackrel{1\text{-}1}{\longleftrightarrow} \{ Q \in \Spec^G(R
\otimes K) \mid Q \cap R = 0 \}
$$
which is equivariant for the $G$-action by $\Id \otimes \rho_\ell$
on $R \otimes K$ and on $\til{R} \otimes K$. Hence it suffices to
construct a bijection
\begin{equation} \label{E:til}
d \colon \Spec_0R \tto \{ Q \in \Spec^G(R \otimes K) \mid Q \cap R =
0 \}
\end{equation}
which is equivariant for the $G$-action by $\rho$ on $R$ and by $\Id
\otimes \rho_\ell$ on $R \otimes K$.

For a given $P \in \Spec_0R$, consider the homomorphism of
$K$-algebras
\begin{equation} \label{E:phiP}
\phi_P \colon R \otimes K \stackrel{\Delta_R\,}{\tto} R \otimes K
\stackrel{\text{can.}}{\onto} S_P = (R/P) \otimes K \ ,
\end{equation}
where $\Delta_R$ is the automorphism \eqref{E:Delta2} and the second
map is the $K$-linear extension of the canonical epimorphism $R
\onto R/P$. The algebra $S_P$ is prime, since $K$ is unirational
over $\k$. Therefore,
\begin{equation} \label{E:d}
d(P) = \Ker \phi_P = \Delta_R^{-1}(P \otimes K)
\end{equation}
is a prime ideal of $R \otimes K$. From  \cite[Lemma 17]{mL08}, we
infer that $P\byG = d(P) \cap R$, and so $d(P) \cap R = 0$.
Furthermore, $d(P)$ clearly determines $P$. Hence, the map $P
\mapsto d(P)$ yields an injection of $\Spec_0R$ into $\{ Q \in
\Spec(R \otimes K) \mid Q \cap R = 0 \}$.

We now check that this injection is $G$-equivariant and has image in
$\Spec^G(R \otimes K)$. Note that \eqref{E:reg1} and \eqref{E:reg2}
imply the following equalities for all $g \in G$:
\begin{align}
\phi_P \circ (\rho \otimes \rho_r)(g) &= (\Id_{R/P} \otimes
\rho_r)(g)\circ \phi_P \,, \label{E:reg1'} \\
\phi_{g.P} \circ (\Id_R \otimes \rho_\ell)(g) &= (\rho \otimes
\rho_\ell)(g)\circ \phi_P \,. \label{E:reg2'}
\end{align}
In \eqref{E:reg2'}, we view $(\rho \otimes \rho_\ell)(g)$ as an
isomorphism $S_P \iso S_{g.P}$ in the obvious way. In particular,
\eqref{E:reg1'} shows that $d(P)$ is stable under $(\rho \otimes
\rho_r)(G)$ while \eqref{E:reg2'} gives
$$
d(g.P) = (\Id_R \otimes \rho_\ell)(g)(P) \ ;
$$
so the map $P \mapsto d(P)$ is $G$-equivariant.

For surjectivity of $d$ and the inverse map, let $Q \in \Spec^G(R
\otimes K)$ be given such that $Q \cap R = 0$. Put
$$
P = R \cap \Delta_R( Q ) \ .
$$
Then $P$ is prime in $R$, because the extension $R \into R \otimes
K$ is centralizing. Furthermore, \cite[Lemma 17]{mL08} gives $P\byG
= \Delta_R^{-1}(P \otimes K) \cap R = Q \cap R = 0$; so $P \in
\Spec_0R$. We claim that $Q = d(P)$ or equivalently, $\Delta_R( Q )
= P \otimes K$. To see this, note that $\Delta_R\left( Q \right)$ is
stable under $(\Id \otimes \rho_r)(G)$ by \eqref{E:reg1}. Thus, the
desired equality $\Delta_R( Q ) = P \otimes K$ follows from
\cite[Cor.~to Prop.~V.10.6]{nB81}, because the field of
$\rho_r(G)$-invariants in $K$ is $\k$. Thus, $d$ is surjective and
the inverse of $d$ is given by
\begin{equation} \label{E:dinv}
d^{-1}(Q) = R \cap \Delta_R( Q ) \ .
\end{equation}
This finishes the construction of the desired $G$-equivariant
bijection \eqref{E:til}.

To summarize, we have constructed a bijection
$$
c \colon \Spec_0R \tto \Spec^G(T_0)
$$
with property (a); it arises as the composite of the following
bijections:
\begin{equation} \label{E:bijections}
\xymatrix{%
\Spec_0R \ar[r]^-{d} \ar[d]_{c = f \circ e^{-1} \circ d} & \{ Q \in
\Spec^G(R \otimes K) \mid Q \cap R = 0 \}\\
\Spec^G(T_0) &  \{\til{Q} \in \Spec^G(\til{R} \otimes K) \mid
\til{Q} \cap \til{R} = 0 \} \ar[u]_e \ar[l]_-f }
\end{equation}
Formulas for $d$ and its inverse are given in \eqref{E:d} and
\eqref{E:dinv}, respectively. The other maps are as follows:
\begin{align}
e(\til{Q}) &= \til{Q} \cap (R \otimes K)\,, \label{E:e} \\
f(\til{Q}) &= \til{Q} \cap T_0  \,, \label{E:f} \\
f^{-1}(\fp) &= \til{R}\otimes_C\fp = (\til{R} \otimes K)\fp \ ;
\end{align}
see Propositions~\ref{L:bijection} and \ref{L:bijection2}. The maps
$d^{\pm 1}$, $f^{\pm 1}$ and $e$ visibly preserve inclusions, and
Proposition~\ref{L:bijection2} tells us that this also holds for
$e^{-1}$. Property (b) follows. By Propositions~\ref{L:bijection}
and \ref{L:bijection2}, both $e$ and $f$ also preserve hearts. Thus,
we have an isomorphism of $K$-fields
\begin{equation*}
\Psi_P \colon \cC(T_0/c(P)) \iso \cC\left((R \otimes K)/d(P)\right)
\iso \cC(S_P) \ .
\end{equation*}
The desired identities for $\Psi$ are consequences of
\eqref{E:reg1'} and \eqref{E:reg2'}. This proves property (c).

For (d), note that
\begin{equation} \label{E:Cincl}
K \subseteq T_0/c(P) \subseteq \cC(T_0/c(P)) = \Fract\left( T_0/c(P)
\right) \cong \cC(S_P)
\end{equation}
holds for any $P \in \Spec_0R$. If $P$ is rational then $\cC(S_P) =
K$ by \cite[Lemma 7]{mL08}, and so $T_0/c(P) = K$. Conversely, if
$T_0/c(P) = K$ then \eqref{E:Cincl} gives $\cC(S_P) = K$. Since
there always is a $K$-embedding $\cC(R/P) \otimes K \into \cC(S_P)$
by \cite[equation (1-2)]{mL08}, we conclude that $\cC(R/P) = \k$; so
$P$ is rational. This proves (d), and hence the proof of the theorem
is complete.
\end{proof}

\subsection{} \label{SS:maximal}

Note that Theorem~\ref{T:fibre}(b) and (d) together imply that
rational ideals are maximal in their $G$-strata. The following
result, which generalizes Vonessen \cite[Theorem 2.3]{nVE98}, is a
marginal strengthening of this fact. In particular, the group $G$ is
no longer assumed to be connected.

\begin{prop} \label{P:maximal}
Let $P \in \Rat R$ and let $I \trianglelefteq R$ be any ideal of $R$
such that $I \supseteq P$. If $I\byG = P\byG$ then $I = P$.
\end{prop}

\begin{proof}
Let $G^0$ denote the connected component of the identity in $G$;
this is a normal subgroup of finite index in $G$ \cite[1.2]{aB91}.
Putting $I^0 = I\!\!:\!\! G^0$ and $P^0 = P\!\!:\!\! G^0$, we have
$$
I^0 \supseteq P^0 \supseteq P\byG = I\byG = \bigcap_{x \in G/G^0}
x.I^0 \ .
$$
Since $P^0$ is $G^0$-prime and all $x.I^0$ are $G^0$-stable ideals
of $R$, we conclude that $I^0 \supseteq P^0 \supseteq x.I^0$ for
some $x$ and \eqref{E:action1} further implies that $I^0 = P^0$.
Therefore, we may replace $G$ by $G^0$, thereby reducing to the case
where $G$ is connected. Furthermore, replacing $I$ by an ideal that
is maximal subject to the condition $I\byG=P\byG$, we may also
assume that $I$ is prime; see \cite[proof of Proposition 8]{mL08}.
Thus, $P$ and $I$ both belong to $\Spec_{P:G}R$ and
Theorem~\ref{T:fibre}(b),(d) yields the result.
\end{proof}

\begin{cor} \label{C:maximal}
Let $I \in \GSpec R$ be locally closed in $\GSpec R$. Then the
maximal members of the $G$-stratum $\Spec_IR$ are locally closed in
$\Spec R$. In particular, if $R$ satisfies the Nullstellensatz (see
\S\ref{SS:DMequiv}) then the maximal members of the $G$-stratum
$\Spec_IR$ are exactly the rational ideals in $\Spec_IR$.
\end{cor}

\begin{proof}
Let $P \in \Spec R$ be maximal in its $G$-stratum $\Spec_IR$, where
$I = P\byG$. Then, for any $Q \in \Spec R$ with $Q \supsetneqq P$,
we have $Q\byG \supsetneqq I$. Since $I$ is locally closed, it
follows that $I \neq \bigcap_{Q \supsetneqq P} Q\byG$. Hence $P \neq
\bigcap_{Q \supsetneqq P} Q$ which proves that $P$ is locally closed
in $\Spec R$.

Finally, rational ideals are always maximal in their $G$-strata by
Proposition~\ref{P:maximal}. In the presence of the Nullstellensatz,
the converse follows from the preceding paragraph.
\end{proof}

\subsection{} \label{SS:nat}

We review some general results of M{\oe}glin and Rentschler
\cite{cMrR86} and Vonessen \cite{nVE98}. Some of the constructions
below were already used, in a more specialized form, in the proof of
Theorem~\ref{T:fibre}. The affine algebraic group $G$ need not be
connected here, but we will only use this material in the connected
case later on.

Fix a closed subgroup $H \le G$ and let
$$
\k(G)^{H} = \k(H\backslash G) \subseteq \k(G)
$$
denote the subalgebra of invariants for the left regular action
$\rho_\ell\big|_{H}$ on $\k(G)$. Following M{\oe}glin and Rentschler
\cite{cMrR86} we define, for an arbitrary ideal $I \trianglelefteq
R$,
\begin{equation} \label{E:nat}
I^\natural = \Delta_R^{-1}(I \otimes \k(G)) \cap (R \otimes \k(G)^H)
\ .
\end{equation}
Here $\Delta_R$ is the automorphism \eqref{E:Delta2} of $R \otimes
\k(G)$. Thus, $I^\natural$ is certainly an ideal of $R \otimes
\k(G)^H$. Furthermore:
\begin{itemize}
\item
If $I$ is semiprime then $I \otimes \k(G)$ is a semiprime ideal of
$R \otimes \k(G)$, because $\k(G)$ is a direct product of fields
that are unirational over $\k$. Therefore, $I^\natural$ is semiprime
in this case. For connected $G$, we also see that if $I$ is prime
then $I^\natural$ is likewise, as in \eqref{E:d}.
\item
The group $G$ acts on $R \otimes \k(G)^H$ by means of $\rho \otimes
\rho_r$. Formula \eqref{E:reg1} implies that $I^\natural$ is always
stable under this action. Moreover, if the ideal $I$ is $H$-stable
then formula \eqref{E:reg2} implies that the ideal $\Delta_R^{-1}(I
\otimes \k(G))$ of $R \otimes \k(G)$ is stable under the
automorphism group $\Id_R \otimes \rho_\ell(H)$. Therefore,
\cite[Lemma 6.3]{nVE98} (or \cite[Cor.~to Prop.~V.10.6]{nB81} for
connected $G$) implies that $\Delta_R^{-1}(I \otimes \k(G)) =
I^\natural \otimes_{\k(G)^H} \k(G)$ holds in this case, and hence
\begin{equation} \label{E:sharp}
I = \Delta_R\left( I^\natural \otimes_{\k(G)^H} \k(G) \right) \cap R
\ .
\end{equation}
\end{itemize}
To summarize, the map $I \mapsto I^\natural$ gives an injection of
the set of all $H$-stable semiprime ideals of $R$ into the set of
all $(\rho \otimes \rho_r)(G)$-stable semiprime ideals of $R \otimes
\k(G)^H$. In fact:

\begin{prop}[M{\oe}glin and Rentschler, Vonessen] \label{P:nat}
Let $H$ be a closed subgroup of $G$. The map $I \mapsto I^\natural$
in \eqref{E:nat} gives a bijection from the set of all $H$-stable
semiprime ideals of $R$ to the set of all $(\rho \otimes
\rho_r)(G)$-stable semiprime ideals of $R \otimes \k(G)^H$. This
bijection and its inverse, given by \eqref{E:sharp}, preserve
inclusions.
\end{prop}

For the complete proof, see \cite[Theorem 6.6(a)]{nVE98}.


\section{Proof of Theorem~\ref{T:orbit}} \label{S:proof}

\subsection{Proof of Theorem~\ref{T:orbit} (b) $\Rightarrow$ (a)}
Let $P \in \Rat R$ be given such that $I = P\byG$ is a locally
closed point of $\GSpec R$. Then the preimage $\gamma^{-1}(I) =
\Spec_IR$ under the continuous map \eqref{E:spec} is a locally
closed subset of $\Spec R$, and hence so is $\Spec_IR \cap \bar{\{ P
\}}$. By Proposition~\ref{P:maximal}, $\Spec_IR \cap \bar{\{ P \}} =
\{ P \}$, which proves (a).

\subsection{Proof of Theorem~\ref{T:orbit} (a) $\Rightarrow$ (b)}

Besides making crucial use of Theorem~\ref{T:fibre}, our proof
closely follows Vonessen \cite[Sect.~8]{nVE98} which in turn is
based on M{\oe}glin and Rentschler \cite[Sect.~3]{cMrR81}.

\subsubsection{Some reductions} \label{SSS:reductions}
First, recall that the connected component of the identity in $G$ is
always a normal subgroup of finite index in $G$. Therefore,
Lemma~\ref{L:finiteindex} allows us to assume that $G$ is connected
and we will do so for the remainder of this section.

We are given an ideal $P \in \Rat R$ satisfying \eqref{E:locclosed}
and our goal is to show that $P\byG$ is distinct from the
intersection of all $G$-primes of $R$ which properly contain
$P\byG$; see \eqref{E:Glocclosed}. For this, we may clearly replace
$R$ by $R/P\byG$. Hence we may assume that
$$
P\byG = 0 \ .
$$
Thus, the algebra $R$ is prime by \cite[Proposition 19(a)]{mL08},
and so $\cC(R)$ is a field. Our goal now is to show that
the intersection of all nonzero $G$-primes of $R$ is nonzero again.
Note that this intersection is identical to the intersection of all
nonzero $G$-stable semiprime ideals of $R$, because $G$-stable
semiprimes are exactly the intersections of $G$-primes. The
intersection in question is also identical to the intersection of
all nonzero $G$-stable prime ideals of $R$, because $G$-primes are
the same as $G$-stable primes of $R$ by \cite[Proposition
19(a)]{mL08}.

\subsubsection{The main lemma} 
Let $\til{R} = R\cC(R)$ denote the central closure of $R$. In the
lemma below, which corresponds to Vonessen \cite[Prop.~8.7]{nVE98},
we achieve our goal for $\til{R}$ in place of $R$:

\begin{lem} \label{L:goalRC}
The intersection of all nonzero $G$-stable semiprime ideals of
$\til{R}$ is nonzero.
\end{lem}

\begin{proof}
Let $G_P$ denote the stabilizer of $P$ in $G$ and recall that $G_P$
is a closed subgroup of $G$ \cite[I.2.12(5)]{jcJ03}. Since $P$ is
locally closed, $P$ is distinct from the intersection of all
$G_P$-stable semiprime ideals of $R$ that properly contain $P$. By
Proposition~\ref{P:nat} with $H = G_P$, we conclude that the ideal
$P^\natural \in \Spec\left( R \otimes K^{G_P} \right)$ is distinct
from the intersection of all $(\rho \otimes \rho_r)(G)$-stable
semiprime ideals of $R \otimes K^{G_P}$ that properly contain
$P^\natural$. Here, we have put $K = \k(G)$ and $K^{G_P} =
\k(G_P\backslash G)$ denotes the invariant subfield of $K$ for the
left regular action $\rho_\ell\big|_{G_P}$ as in \S\ref{SS:nat}. In
other words, the intersection of all nonzero $(\rho \otimes
\rho_r)(G)$-stable semiprime ideals of $(R \otimes
K^{G_P})/P^\natural$ is nonzero. By Vonessen \cite[Lemma
8.6]{nVE98}, the lemma will follow if we can show that there is a
finite centralizing embedding
\begin{equation} \label{E:embed}
\til{R} \into (R \otimes K^{G_P})/P^\natural
\end{equation}
such that the $G$-action via $\rho \otimes \rho_r$ on $(R \otimes
K^{G_P})/P^\natural$ restricts to the $G$-action on $\til{R}$ via
the unique extension of $\rho$.

To construct the desired embedding, write $C = \cC(R)$  and consider
the $\k$-algebra map
\begin{equation} \label{E:psiP}
\psi_P \colon C \into T_0 = C \otimes K \onto T_0/c(P)
\iso K \ ,
\end{equation}
where the first two maps are canonical and the last map is the
$K$-isomorphism in Theorem~\ref{T:fibre}(d); it is given by the
isomorphism $\Psi_P$ in Theorem~\ref{T:fibre}(c). (We remark that
the map $\psi_P$ is identical to the one constructed in
\cite[Theorem 22]{mL08}.) The identities for $\Psi_P$ in
Theorem~\ref{T:fibre}(c) yield the following formulas:
\begin{align}
\psi_P \circ \rho(g)\big|_{C} &= \rho_r(g)\circ \psi_P\,,
\label{E:reg1''} \\
\psi_{g.P} &= \rho_\ell(g)\circ \psi_P\,. \label{E:reg2''}
\end{align}
(See also (a) in the proof of \cite[Theorem 22]{mL08}.) Consider the
subfield
\begin{equation*}
K_P = \Im \psi_P \subseteq K \ .
\end{equation*}
Equation \eqref{E:reg1''} implies that $K_P$ is a $\rho_r(G)$-stable
subfield of $K$. More precisely, identity \eqref{E:reg2''} shows
that
$$
K_P \subseteq K^{G_P}  \ .
$$
Moreover, if $g \notin G_P$ then $c(g.P) \neq c(P)$ and hence
$\psi_{g.P} \neq \psi_P$. Therefore, we deduce from \eqref{E:reg2''}
that $G_P = \{ g \in G \mid \rho_\ell(g)(x) = x \text{ for all $x
\in K_P$}\}$. Now Vonessen \cite[Proposition 4.5]{nVE98} gives that
the field extension $K^{G_P}/K_P$ is finite (and purely
inseparable). Thus, the desired embedding \eqref{E:embed} will
follow if we can show that there is a $(\rho \otimes
\rho_r)(G)$-equivariant isomorphism
$$
(R \otimes K^{G_P})/P^\natural \cong \til{R} \otimes_C K^{G_P}\ ,
$$
where $c \otimes 1 = 1 \otimes \psi_P(c)$ holds for all $c \in C$ in
the ring on the right. But the map $\til{R} \otimes K \onto \til{R}
\otimes_C K \cong \til{R} \otimes_C (T_0/c(P))$ has kernel
$(e^{-1}\circ d)(P)$ in the notation of \eqref{E:bijections}, and it
is $(\rho \otimes \rho_r)(G)$-equivariant. The restriction of this
map to $R \otimes K^{G_P}$ has image $\til{R} \otimes_C K^{G_P}$ and
kernel $(e^{-1}\circ d)(P) \cap (R \otimes K^{G_P}) = P^\natural$.
This finishes the proof of the lemma.
\end{proof}

\subsubsection{End of proof} 
We now complete the proof of Theorem~\ref{T:orbit} by showing that
the intersection of all nonzero $G$-stable prime ideals of $R$ is
nonzero.

We use the $G$-equivariant embedding $\psi_P \colon C \into K =
\k(G)$; see \eqref{E:psiP} and \eqref{E:reg1''}. Note that $K =
\Fract A$ where $A = \k[G]$ is a $G$-stable affine domain over $\k$
whose maximal ideals form one $G$-orbit. Therefore, $A$ is
$G$-simple. By Proposition~\ref{L:VE}, there exists an ideal $0 \neq
D \trianglelefteq R$ such that, for every $G$-stable prime $P$ of
$R$ not containing $D$, there is a $G$-stable prime of $\til{R}$
lying over $P$.

Now let $N$ denote the intersection of all nonzero $G$-stable
semiprime ideals of $\til{R}$; so $N \neq 0$ by Lemma~\ref{L:goalRC}
and hence $N \cap R \neq 0$ by \cite[Prop.~2(ii)(iii)]{mL08}. We
conclude from the preceding paragraph that every nonzero $G$-stable
prime ideal of $R$ either contains $D$ or else it contains $N \cap
R$. Therefore, the intersection of all nonzero $G$-stable prime
ideals of $R$ contains the ideal $D \cap N \cap R$ which is nonzero,
because $R$ is prime. This completes the proof of
Theorem~\ref{T:orbit}. \qed


\section{Finiteness of $\GSpec R$} \label{S:finiteness}

\subsection{} \label{SS:finiteness}
The following finiteness result is an application of
Theorem~\ref{T:orbit} 
and \cite[Theorem 1]{mL08}.

\begin{prop} \label{P:finiteness}
Assume that $R$ satisfies the Nullstellensatz. Then the following
conditions are equivalent:
\begin{enumerate}
\item $R$ has finitely many $G$-stable semiprime ideals;
\item $\GSpec R$ is finite;
\item $\GRat R$ is finite;
\item $G$ has finitely many orbits in $\Rat R$;
\item  $R$ satisfies
    \begin{minipage}[t]{4.6in}
    \begin{enumerate}
    \item[(i)] the ascending chain condition for $G$-stable semiprime ideals,
    \item[(ii)] the Dixmier-M{\oe}glin equivalence, and
    \item[(iii)] $\GRat R = \GSpec R$.
    \end{enumerate}
    \end{minipage}
\end{enumerate}
If these conditions are satisfied then rational ideals of $R$ are
exactly the primes that are maximal in their $G$-strata.
\end{prop}

\begin{proof}
The implications (a) $\Rightarrow$ (b) $\Rightarrow$ (c) are trivial
and (c) $\Leftrightarrow$ (d) is clear from \eqref{E:rat}. Moreover,
the Nullstellensatz implies that the  $G$-stable semiprime ideals of
$R$ are exactly the intersections of $G$-rational ideals. Thus, (c)
implies (a), and hence conditions (a) - (d) are all equivalent.

We now show that (a) - (d) imply (e).  First, (i) is trivial from
(a). For (ii), note that (b) implies that all points of $\GSpec R$
are locally closed. Hence all rational ideals of $R$ are locally
closed in $\Spec R$ by Theorem~\ref{T:orbit}, proving (ii). Finally,
in order to prove (iii), write a given $I \in \GSpec R$ as an
intersection of $G$-rational ideals. The intersection involves only
finitely many members by (c), and so one of them must be equal to
$I$ by $G$-primeness. Thus, $I \in \GRat R$ which takes care of
(iii).

To complete the proof of the equivalence of (a) - (e), we will show
that (e) implies (b). By a familiar argument, hypothesis (i) allows
us to assume that the algebra $R$ is $G$-prime and $\GSpec R/I$ is
finite for all nonzero $G$-stable semiprime ideals $I$ of $R$. By
(iii) and \eqref{E:rat}, we know that $P\byG = 0$ holds for some $P
\in \Rat R$. Since $P$ is locally closed in $\Spec R$ by (ii),
Theorem~\ref{T:orbit} implies that $0$ is locally closed in $\GSpec
R$, that is, $I = \bigcap_{0 \neq J \in \GSpec R } J$ is nonzero;
see \eqref{E:Glocclosed}. Therefore, $\GSpec R = \{0\} \cup \GSpec
R/I$ is finite.

Finally, the last assertion is clear from Corollary~\ref{C:maximal},
because all points of $\GSpec R$ are locally closed by (b).
\end{proof}

\subsection{} \label{SS:torus}
We now concentrate on the case where $G$ is an algebraic torus: $G
\cong \Gm^n$ with $\Gm = \k^*$. In particular, $G$ is connected and
so $\GSpec R$ is simply the set of all $G$-stable primes of $R$ by
\cite[Proposition 19(a)]{mL08}. Moreover, every $G$-module $M$ has
the form
$$
M = \bigoplus_{\lambda \in X(G)} M_\lambda \ ,
$$
where $X(G)$ is the collection of all morphisms of algebraic groups
$\lambda \colon G \to \Gm$ 
and
$$
M_\lambda = \{ m \in M \mid g.m = \lambda(g)m \text{ for all $g \in
G$}\}
$$
is the set of $G$-eigenvectors of weight $\lambda$ in $M$.

\begin{lem} \label{L:1}
If $\dim_\k R_\lambda \le 1$ holds for all $\lambda \in X(G)$ then
$\GSpec R = \GRat R$.
\end{lem}

\begin{proof}
Let $P \in \GSpec R$ be given. The condition $\dim_\k R_\lambda \le
1$ for all $\lambda \in X(G)$ passes to $R/P$. Therefore, replacing
$R$ by $R/P$, we may assume that $R$ is prime and we must show that
$\cC(R)^G = \k$. For a given $q \in \cC(R)^G$ put $I = \{ r \in R
\mid qr \in R \}$; this is a nonzero $G$-invariant ideal of $R$.
Hence $I = \bigoplus_{\lambda \in X(G)} I_\lambda$ and so there is a
nonzero element $r \in I_\lambda$ for some $\lambda$. Since $q$ is
$G$-invariant, we have $qr \in R_\lambda = \k r$. Therefore, $(q-k)r
= 0$ holds in the central closure $\til{R}$ for some $k \in \k$.
Inasmuch as nonzero elements of $\cC(R)$ are units in $\til{R}$, we
conclude that $q = k \in \k$, which proves the lemma.
\end{proof}

\subsubsection{Example: affine  commutative algebras} \label{SSS:comm}
The following proposition is a standard result on $G$-varieties
 \cite[II.3.3 Satz 5]{hK84}. It is also immediate from the foregoing:

\begin{prop} \label{P:comm}
Let $R$ be an affine commutative domain over $\k$ and let $G$ be an
algebraic $\k$-torus acting rationally on $R$. Then the following
are equivalent:
\begin{enumerate}
\item[(i)] $\GSpec R$ is finite;
\item[(ii)] $(\Fract R)^G = \k$;
\item[(iii)] $\dim_\k R_\lambda \le 1$ for all $\lambda \in X(G)$.
\end{enumerate}
\end{prop}

\begin{proof}
Since affine commutative algebras satisfy the Nullstellensatz, the
Dixmier-M{\oe}glin equivalence and the ascending chain condition for
ideals, Proposition~\ref{P:finiteness} tells us that (i) amounts to
the equality $\GSpec R = \GRat R$. The implication (iii)
$\Rightarrow$ (i) therefore follows from Lemma~\ref{L:1}.
Furthermore, (i) implies that $0 \in \GRat R$; so $\cC(R)^G = \k$.
Since $\cC(R) = \Fract R$, (ii) follows. Finally, if $a,b \in
R_\lambda$ are linearly independent then $a/b \in \Fract R$ is not a
scalar. Hence (ii) implies (iii).
\end{proof}

\subsubsection{Example: quantum affine toric varieties} \label{SSS:qtoric}
Affine domains $R$ with a rational action by an algebraic torus $G$
satisfying the condition $\dim_\k R_\lambda \le 1$ for all $\lambda
\in X(G)$ as in Lemma~\ref{L:1} are called \emph{quantum affine
toric varieties} in Ingalls \cite{cIxx}.

A particular example is \emph{quantum affine space} $R =
\cO_\bq(\k^n) = \k[x_1,\dots,x_n]$. Here, $\bq$ denotes family of
parameters $q_{ij} \in \k^*$ $(1 \le i < j \le n)$ and the defining
relations of $R$ are
$$
x_j x_i = q_{ij} x_i x_j \qquad (i < j) \ .
$$
The torus $G = \Gm^n$ acts on $R$, with $\alpha =
(\alpha_1,\dots,\alpha_n) \in G$ acting by
$$
\alpha.x_i = \alpha_i x_i
$$
for all $i$. The standard PBW-basis of $R$,
$$
\{ x^\bfm = x_1^{m_1}x_2^{m_2}\dots x_n^{m_n} \mid \bfm =
(m_1,\dots,m_n) \in \ZZ_{\ge 0}^n \}\ ,
$$
consists of $G$-eigenvectors: $x^\bfm \in R_{\lambda_\bfm}$ with
$$
\lambda_\bfm(\alpha) = \alpha^\bfm =
\alpha_1^{m_1}\alpha_2^{m_2}\dots \alpha_n^{m_n}
$$
Therefore, the condition $\dim_\k R_\lambda \le 1$ for all $\lambda$
is satisfied. Moreover, $R = \cO_\bq(\k^n)$ is noetherian and
satisfies the Dixmier-Moeglin equivalence as long as $\k$ contains a
non-root of unity \cite[II.8.4]{kBkG02}.

Any quantum affine toric variety $R$ is a quotient of some
$\cO_\bq(\k^n)$ \cite[p.~6]{cIxx}. Hence, $R$ is noetherian and
satisfies the Dixmier-Moeglin equivalence (in the presence of
non-roots of unity). Therefore, $\GSpec R$ is finite by
Proposition~\ref{P:finiteness} and Lemma~\ref{L:1}.

\subsubsection{Example: quantum $2 \times 2$ matrices} \label{SSS:q2x2}
Let $R = \cO_q(\Mat_2)$; this is the algebra with generators $a, b,
c, d$ and defining relations
\begin{align*}
ab &= qba & ac &= qca & bc &= cb \\
bd &= qdb & cd &= qdc & ad - da &= (q - q^{-1})bd\ .
\end{align*}
The torus $\Gm^4$ acts on $R$ as in \cite[II.1.6(c)]{kBkG02}, with
$D= \{ (\alpha,\alpha,\alpha^{-1},\alpha^{-1}) \mid \alpha \in \k^*
\}$ acting trivially. Thus, $G = \Gm^4/D \cong \Gm^3$ acts on $R$:
\begin{align*}
(\alpha,\beta,\gamma).a &= \beta a & (\alpha,\beta,\gamma).b &=
\gamma b  \\
(\alpha,\beta,\gamma).c &= \alpha\beta c & (\alpha,\beta,\gamma).d
&= \alpha\gamma d
\end{align*}
This action does not satisfy condition $\dim_\k R_\lambda \le 1$ for
all $\lambda \in X(G)$. Indeed, the PBW-basis $\{ a^ib^jc^ld^m \mid
i,j,l,m \in \ZZ_{\ge 0} \}$ of $R$ consists of $G$-eigenvectors:
 $a^ib^jc^ld^m$ corresponds to the character
$(\alpha,\beta,\gamma) \mapsto \alpha^{l+m}\beta^{i+l}\gamma^{j+m}$.
Defining
$$
f \colon \ZZ^4 \tto \ZZ^3\ ,\qquad (i,j,l,m) \mapsto (l+m,i+l,j+m)
$$
we have, for any given $\lambda = (\lambda_1, \lambda_2,\lambda_3)
\in \ZZ^3$,
$$
\dim_\k R_\lambda = \# \left( f^{-1}(\lambda) \cap \ZZ_{\ge 0}^4
\right) \ .
$$
In order to determine this number, note that $\Ker f = \ZZ
(1,-1,-1,1)$. Hence, we must count the possible $t \in \ZZ$ so that
$z_\lambda + t(1,-1,-1,1) \in \ZZ_{\ge 0}^4$ where we have put
$z_\lambda = (\lambda_2 - \lambda_1, \lambda_3, \lambda_1, 0)$. We
obtain the following conditions on $t$: $\lambda_2 - \lambda_1 + t
\ge 0$, $\lambda_3 - t \ge 0$, $\lambda_1 - t \ge 0$, and $t \ge 0$.
Therefore,
\begin{equation} \label{E:dim}
\dim_\k R_\lambda = \max \left\{ 0, \min \{ \lambda_1, \lambda_3 \}
- \max \{ \lambda_1 - \lambda_2,0 \} + 1 \right\}
\end{equation}
which can be arbitrarily large.

Now consider the algebra $\bar{R} = R/(D_q)$ where $D_q = ad - q bc
\in \cen R$ is the quantum determinant; see \cite[I.1.9]{kBkG02}.
Since $D_q \in R_\pi$ with $\pi = (1,1,1)$, we have
$$
\dim_\k \bar{R}_\lambda = \dim_\k R_\lambda - \dim_\k R_{\lambda -
\pi} \le 1
$$
by \eqref{E:dim}. Therefore $\dim_\k \bar{R}_\lambda \le 1$ for all
$\lambda \in X(G)$. Moreover, assuming $q$ to be a non-root of
unity, $\bar{R}$ satisfies the Nullstellensatz and the
Dixmier-M{\oe}glin equivalence. Indeed, the algebra $\bar{R}$ is an
image of quantum $4$-space, since $ad \equiv q^2 da \bmod D_q$.
Therefore, we know from Proposition~\ref{P:finiteness} and
Lemma~\ref{L:1} that there are finitely many $G$-primes of $R$ that
contain $D_q$. The remaining $G$-primes correspond to $\GSpec
\cO_q(\GL_2)$, and by \cite[Exercise II.2.M]{kBkG02}, there are only
four of these.


\begin{ack}
The author would like to thank Temple University for granting him a
research leave during the Fall Semester 2008 when the work on this
article was completed. Thanks are also due to the referees for
valuable comments and suggestions.
\end{ack}



\begin{thebibliography}{10}

\bibitem{gBxx}
George~M. Bergman, \emph{More on extensions by centralizing
elements}, unpublished preprint.

\bibitem{aB91}
Armand Borel, \emph{Linear algebraic groups}, second ed., Graduate
Texts in Mathematics, vol. 126, Springer-Verlag, New York, 1991.

\bibitem{BGR73}
Walter Borho, Peter Gabriel, and Rudolf Rentschler, \emph{Primideale
in {E}inh\"ullenden aufl\"osbarer {L}ie-{A}lgebren ({B}eschreibung
durch {B}ahnenr\"aume)}, Springer-Verlag, Berlin, 1973, Lecture
Notes in Mathematics, Vol. 357.

\bibitem{wBrR06}
Walter Borho and Rudolf Rentschler, \emph{Sheets and hearts of prime
ideals in enveloping algebras of semisimple {L}ie algebras}, J.
Algebra \textbf{304} (2006), no.~1, 324--348.

\bibitem{nB71}
N.~Bourbaki, \emph{\'{E}l\'ements de math\'ematique. {T}opologie
g\'en\'erale. {C}hapitres 1 \`a 4}, Hermann, Paris, 1971.

\bibitem{nB81}
Nicolas Bourbaki, \emph{Alg\`ebre, chapitres 4 \`a 7}, Masson,
Paris, 1981.

\bibitem{kBkG02}
Kenneth~A. Brown and Kenneth~R. Goodearl, \emph{Lectures on
algebraic quantum groups}, Advanced Courses in Mathematics. CRM
Barcelona, Birkh\"auser Verlag, Basel, 2002.

\bibitem{jD96}
Jacques Dixmier, \emph{Enveloping algebras}, Graduate Studies in
Mathematics, vol.~11, American Mathematical Society, Providence, RI,
1996, Revised reprint of the 1977 translation.

\bibitem{EMO75}
Theodore~S. Erickson, Wallace~S. Martindale, III, and J.~Marshall
Osborn, \emph{Prime nonassociative algebras}, Pacific J. Math.
\textbf{60} (1975),
  no.~1, 49--63.

\bibitem{pG71}
Pierre Gabriel, \emph{Repr\'esentations des alg\`ebres de {L}ie
r\'esolubles
  (d'apr\`es {D}ixmier)}, S\'eminaire Bourbaki, 21\`eme ann\'ee, 1968/69,
  expos\'e no 347, Lecture Notes in Mathematics, vol. 179, Springer-Verlag, New
  York, 1971, pp.~1--22.

\bibitem{kGeL00}
K.~R. Goodearl and E.~S. Letzter, \emph{The {D}ixmier-{M}oeglin
equivalence in quantum coordinate rings and quantized {W}eyl
algebras}, Trans. Amer. Math. Soc. \textbf{352} (2000), no.~3,
1381--1403.

\bibitem{cIxx}
Colin Ingalls, \emph{Quantum toric varieties}, preprint.

\bibitem{jcJ03}
Jens~Carsten Jantzen, \emph{Representations of algebraic groups},
second ed., Mathematical Surveys and Monographs, vol. 107, American
Mathematical Society, Providence, RI, 2003.

\bibitem{hK84}
Hanspeter Kraft, \emph{Geometrische {M}ethoden in der
{I}nvariantentheorie}, Aspects of Mathematics, D1, Friedr. Vieweg \&
Sohn, Braunschweig, 1984.


\bibitem{mL08}
Martin Lorenz, \emph{Group actions and rational ideals}, Algebra and
Number Theory \textbf{2} (2008), no.~4, 467--499, available at
arXiv:0801.3472.

\bibitem{wM69}
Wallace~S. Martindale, III, \emph{Prime rings satisfying a
generalized
  polynomial identity}, J. Algebra \textbf{12} (1969), 576--584.

\bibitem{jM82}
J.~Matczuk, \emph{Central closure of semiprime tensor products},
Comm. Algebra
  \textbf{10} (1982), no.~3, 263--278.

\bibitem{jMcCjR87}
J.~C. McConnell and J.~C. Robson, \emph{Noncommutative {N}oetherian
rings}, revised ed., Graduate Studies in Mathematics, vol.~30,
American Mathematical Society, Providence, RI, 2001, With the
cooperation of L. W. Small.

\bibitem{cMrR81}
Colette M{\oe}glin and Rudolf Rentschler, \emph{Orbites d'un groupe
  alg\'ebrique dans l'espace des id\'eaux rationnels d'une alg\`ebre
  enveloppante}, Bull. Soc. Math. France \textbf{109} (1981), no.~4, 403--426.


\bibitem{cMrR84}
\bysame, \emph{Sur la classification des id\'eaux primitifs des
alg\`ebres enveloppantes}, Bull. Soc. Math. France \textbf{112}
(1984), no.~1, 3--40.


\bibitem{cMrRxx}
\bysame, \emph{Id{\'e}aux {$G$}-rationnels, rang de {G}oldie},
preprint, 1986.

\bibitem{cMrR86}
\bysame, \emph{Sous-corps commutatifs ad-stables des anneaux de
fractions des quotients des alg\`ebres enveloppantes; espaces
homog\`enes et induction de {M}ackey}, J. Funct. Anal. \textbf{69}
(1986), no.~3, 307--396.

\bibitem{dM88}
David Mumford, \emph{The red book of varieties and schemes}, Lecture
Notes in Mathematics, vol. 1358, Springer-Verlag, Berlin, 1988.

\bibitem{cP73}
Claudio Procesi, \emph{Rings with polynomial identities}, Marcel
Dekker Inc., New York, 1973, Pure and Applied Mathematics, 17.

\bibitem{rR87}
Rudolf Rentschler, \emph{Primitive ideals in enveloping algebras
(general case)}, Noetherian rings and their applications
(Oberwolfach, 1983), Math.
  Surveys Monogr., vol.~24, Amer. Math. Soc., Providence, RI, 1987, pp.~37--57.


\bibitem{lR88}
Louis~H. Rowen, \emph{Ring theory. {V}ol. {II}}, Pure and Applied
Mathematics, vol. 128, Academic Press Inc., Boston, MA, 1988.

\bibitem{nVE96}
Nikolaus Vonessen, \emph{Actions of algebraic groups on the spectrum
of rational ideals}, J. Algebra \textbf{182} (1996), no.~2,
383--400.


\bibitem{nVE98}
\bysame, \emph{Actions of algebraic groups on the spectrum of
rational ideals. {II}}, J. Algebra \textbf{208} (1998), no.~1,
216--261.

\end{thebibliography}

\def\cprime{$'$}
\providecommand{\bysame}{\leavevmode\hbox
to3em{\hrulefill}\thinspace}
\providecommand{\MR}{\relax\ifhmode\unskip\space\fi MR }
\providecommand{\MRhref}[2]{%
  \href{http://www.ams.org/mathscinet-getitem?mr=#1}{#2}
} \providecommand{\href}[2]{#2}


\end{document}